\documentclass[a4paper,12pt]{amsart}
\usepackage{amsmath}
\usepackage{amsthm}
\usepackage{amsfonts}
\usepackage{amssymb}
\usepackage{multicol}
\usepackage{enumerate}
\usepackage{tikz}
\usepackage{tikz-cd}
\usetikzlibrary{cd}
\usepackage[unicode]{hyperref}

\newtheorem{thm}{Theorem}[section]
\newtheorem{cor}[thm]{Corollary}
\newtheorem{lemma}[thm]{Lemma}
\newtheorem{prop}[thm]{Proposition}

\theoremstyle{definition}
\newtheorem{remark}[thm]{Remark}

\numberwithin{equation}{section}

\def\al{\alpha}
\def\de{\delta}
\def\la{\lambda}
\def\si{\sigma}
\def\Om{\Omega}
\def\R{\mathbb{R}}
\def\C{\mathbb{C}}
\def\N{\mathbb{N}}
\def\J{\mathbb{J}}
\def\Tr{\mathrm{Tr}\,}
\def\Ker{\mathrm{Ker}\,}
\def\Ran{\mathrm{Ran}\,}
\def\Id{\mathbf{1}}
\def\cB{\mathcal{B}}
\def\cF{\mathcal{F}}
\def\cG{\mathcal{G}}
\def\cN{\mathcal{N}}
\def\cU{\mathcal{U}}
\def\cL{\mathcal{L}}
\def\supp{\mathrm{supp}}

\title[MOP related to difference operators]{Matrix orthogonal polynomials related to diffference operators for $L$-orthogonal polynomials}
\author{Erik Koelink}
\address{IMAPP, Radboud Universiteit, Nijmegen, The Netherlands}
\email{e.koelink@math.ru.nl}
\author{Ana Loureiro}
\address{School of Engineering, Mathematics and Physics, University of Kent – London and South East
University Group, Canterbury CT2 7NF, UK}
\email{a.loureiro@kent.ac.uk}
\author{Pablo Rom\'an}
\address{FaMAF-CIEM, Universidad Nacional de C\'ordoba, Argentina}
\email{pablo.roman@unc.edu.ar}

\date{\today}

\begin{document}

\begin{abstract} The difference operator $J$ arising from the higher order occurrence for the $L$-orthogonal polynomials is studied using various families of matrix orthogonal polynomials. The matrix orthogonal polynomials determine the polar decomposition of $J$ and its adjoint $J^\ast$ considered as operators on $\ell^2(\N)$. The special case of symmetric 
$J$ leads to the case where the matrix polynomials reduce to diagonal matrix polynomials. The example of 
the symmetric case with constant coefficients for $L=2$ is discussed in detail. 
\end{abstract}

\maketitle

%%%%%%%%%%%%%%%%%%%%%%%%%%%%%%%%%
%%%%%%%%%%%%%%%%%%%%%%%%%%%%%%%%%
%%%%%%%%%%%%%%%%%%%%%%%%%%%%%%%%%
%%%%%%%%%%%%%%%%%%%%%%%%%%%%%%%%%

\section{Introduction}\label{sec:intro}

Orthogonal polynomials on the real line, say $(p_n)_{n\in \N}$, satisfy a three-term recurrence relation of the form 
\begin{equation}\label{eq:3TR-OPorthonormal}
\begin{split}
x\, p_n(x) &\, = a_n p_{n+1}(x) + b_n p_n(x) + a_{n-1} p_{n-1}(x), \qquad n\geq 1, \\
x\, p_0(x) &\, = a_0 p_{1}(x) + b_0 p_0(x),
\end{split}
\end{equation}
where $a_n>0$ and $b_n\in \R$ for all $n\in \N$ in case the polynomials are normalised to be orthonormal. Various properties of the  polynomials can then be studied by understanding the operator $J$ defined on the Hilbert space $\ell^2(\N)$ by 
\begin{equation}\label{eq:Jacobioperator-OPorthonormal}
\begin{split}
\bigl(Jf\bigr) (n) &\, = a_n f(n+1) + b_n f(n) + a_{n-1} f(n-1), \qquad n\geq 1, \\
\bigl(Jf\bigr) (0) &\, = a_0 f(1) + b_0 f(0),
\end{split}
\end{equation}
where we view an element $f\in \ell^2(\N)$ as a function   $f\colon \N\to \C$ with 
$\sum_{n\in \N} |f(n)|^2<\infty$. Assuming that the sequences 
$(a_n)_{n\in \N}$ and $(b_n)_{n\in \N}$ are bounded, $J$ is a bounded self-adjoint operator which can be considered as an infinite tridiagonal matrix. The operator $J$ contains much  information on the polynomials, such as the orthogonality measure, and can be used to prove Favard's theorem, stating that polynomials satisfying a recurrence of the 
form \eqref{eq:3TR-OPorthonormal} with initial value $p_0(x)=1$, are orthogonal polynomials on the real line. For more information we refer to \cite{Chih}, \cite{DunfS}, \cite{Isma}, \cite{Schm}, and review papers \cite{Koel}, \cite{Simo},   
where one can also find the case of unbounded sequences and other extensions. 

A natural question is to ask what happens if one drops the conditions $a_n>0$, $b_n\in \R$ in \eqref{eq:3TR-OPorthonormal}, and consider the analogue  
\begin{equation}\label{eq:general3TR}
x\, p_n(x) = a_n p_{n+1}(x) + b_n p_n(x) + c_n p_{n-1}(x), \end{equation}
for more general sequences $(a_n)_{n\in \N}$, $(b_n)_{n\in \N}$ and $(c_n)_{n\in \N}$. We use the convention $c_0=0$ in order to avoid writing down the case $n=0$ separately. Since we want polynomials of any degree, we need to assume $a_n\not=0$ for all $n\in \N$. By multiplying with a phase factor, we may still assume $a_n>0$ for all $n\in\N$. We can associate similarly a tridiagonal operator $J$ on $\ell^2(\N)$ to the recurrence \eqref{eq:general3TR}. 
In case $c_n=a_{n-1}$ and $b_n\in \C$, the corresponding operator $J$ and its spectral properties have been studied in detail by Pushnitski and \v{S}tampach \cite{PushS} in the bounded case and extended by 
Eichinger, Luki\'c and Young \cite{EichLY} to the unbounded case. The essential part is that complex conjugation intertwines $J$ and $J^\ast$ in this case. For the general case of complex Jacobi matrices related to the general case of \eqref{eq:general3TR} we refer to the overview paper by 
Beckermann \cite{Beck}. See also the references in \cite{Beck},
\cite{EichLY}, \cite{PushS}. 

For this paper two generalisations of \eqref{eq:3TR-OPorthonormal} are relevant. The first one is matrix orthogonal polynomials on the real line with respect to a matrix valued measure in the positive definite matrices. In this case there is a three-term recurrence as in \eqref{eq:3TR-OPorthonormal}, but now the coefficients $a_n$ and $b_n$ are $N\times N$-matrices and the polynomials $p_n$ take their values in the $N\times N$-matrices as well. In the corresponding case the corresponding operator can be viewed as a self-adjoint operator acting on $\ell^2(\N;\C^N)$, i.e. square summable functions taking their values in the $N$-dimensional vector space $\C^N$. Its spectral decomposition can be used to analyse 
self-adjoint $2N+1$-recurrence relations, extending the case $N=1$ which corresponds to \eqref{eq:3TR-OPorthonormal}. This is due to 
Dur\'an and Van Assche \cite{DuraVA}. Matrix orthogonal polynomials appear in various other fields of mathematics as well, and we use the overview paper by Damanik, Pushnitski and Simon \cite{DamaPS} as a reference. The paper \cite{DamaPS} contains a wealth of references up to 2008, including references to pioneering work by Krein. Since the notation for matrix orthogonal polynomials and various results are important, we present a brief overview of matrix orthogonality and matrix polynomials in Section \ref{sec:recapMOP}.  

The second generalisation of  \eqref{eq:3TR-OPorthonormal} relevant for this paper is the extension to more terms involving lower degree polynomials, the so-called $L$-orthogonal polynomials. So from now on the polynomial sequence $(p_n)_{n\in \N}$ is generated by 
\begin{equation}\label{eq:recursionscalarLorthopols}
x\, p_n(x) = a_n p_{n+1}(x) + \sum_{k=0}^L b_{n,k} p_{n-k}(x) 
\end{equation}
with initial value $p_0(x)=1$. Here $L$ is fixed and the case 
$L=1$ gives a three-term recurrence, but not necessarily a self-adjoint one. 
We assume that $a_n>0$ for all $n\in \N$, $b_{n,k}=0$ for $n<k\leq L$
and $b_{n,L}\not=0$ for all $n\geq L$. This assumption avoids needing to split up the recurrence \eqref{eq:recursionscalarLorthopols} for general $n$ and the initial $L$ cases. In case $L=1$, we arrive at \eqref{eq:general3TR}. 
Note that $L\in \N$, and in general we assume $L>0$, although we discuss the case $L=0$ in 
Remark \ref{rmk:caseLis0:sec:MOPrelatedtoJ} to connect to $LU$-factorisations of Jacobi matrices and in Remark \ref{rmk:caseLis0:sec:relationbetweenMOP}. . 
The polynomials defined by \eqref{eq:recursionscalarLorthopols} are known as $L$-orthogonal polynomials, as studied by Maroni \cite{Maro}, and often referred to as vector orthogonal polynomials following the work of Van Iseghem \cite{VIse87}. These polynomials fit into the multiple orthogonal polynomials that are important in Hermite-Pad\'e approximation, see e.g. the overview papers by L\'opez-Lagomasino \cite{LopeL} and Van Assche \cite[Ch.~23]{Isma} and references therein. 

So we now fix the corresponding operator $J$ for the remainder of the paper. Initially, we define $J$ on more general function spaces before restricting to suitable Hilbert spaces. 
So let $\cF(\N) = \{ f\colon \N \to \C\}$ and more generally 
$\cF(\N;\C^N) = \{ f\colon \N \to \C^N\}$ be the space of functions and $\C^N$-vector-valued functions on $\N$. Associated to \eqref{eq:recursionscalarLorthopols} is the linear operator
$J\colon \cF(\N) \to \cF(\N)$ defined by 
\begin{equation}\label{eq:defoperatorJassoctopols}
\bigl( Jf \bigr) (n) =  a_n f(n+1) + \sum_{k=0}^L b_{n,k} f(n-k)
\end{equation}
for $f\in \cF(\N)$. We assume that the sequences are bounded, i.e. 
$\sup_{n\in \N} \bigl( |a_n| + \sum_{k=0}^L |b_{n,k}|\bigr) <\infty$, so that $J$ defines a bounded operator when restricted to $\ell^2(\N)$. Its adjoint gives rise to 
$J^\ast \colon \cF(\N) \to \cF(\N)$ defined by 
\begin{equation}\label{eq:defoperatorJstarassoctopols}
\bigl( J^\ast f \bigr) (n) =  \sum_{k=0}^L \overline{b_{n+k,k}} f(n+k)
+ a_{n-1} f(n-1)
\end{equation}
for $f\in \cF(\N)$,  where we follow the convention that $a_{-k}=0$ and $f(-k)=0$ for $k\geq 1$ and $f\in \cF(\N)$. 
Note that for $L>1$, the map $J$ can not be a self-adjoint operator. Moreover, there is no straightforward connection between $J$ and $J^\ast$ as in \cite{PushS}, \cite{EichLY}. 

The first goal of the paper is to connect the map $J$ to three related families of matrix orthogonal polynomials of size $(L+1)\times(L+1)$. 
This is done in Section \ref{sec:MOPrelatedtoJ} after rewriting 
$J$ and $J^\ast$ in terms of difference operators on 
$\cF(\N;\C^{L+1})$ in Section \ref{sec:differenceoperators}. 
The matrix orthogonal polynomials are associated to 
$\begin{pmatrix} 0  & J \\ J^\ast & 0 \end{pmatrix}$, $JJ^\ast$ and $J^\ast J$. 
In Section \ref{sec:relationbetweenMOP} we relate these three families of matrix orthogonal polynomials as well as the Stieltjes transforms of the matrix orthogonality measures, or the Weyl $M$-matrix function for these measures. In 
Section \ref{sec:polardecompositions} we present the polar decompositions of $J$ and $J^\ast$ in full detail. The partial isometry gives rise to a connection between the Stieltjes transforms, which can be viewed as matrix analogues of well-known connections for symmetric scalar orthogonal polynomials. 

In the last two sections we consider two special cases. In 
Section \ref{sec:symmetriccase} we consider the symmetric case of 
the $L$-orthogonal polynomials \eqref{eq:recursionscalarLorthopols}, which means that 
$b_{n,k}=0$ unless $k=L$ so that only two terms remain in the right hand side of \eqref{eq:recursionscalarLorthopols}. For $L=1$, this corresponds to symmetric scalar orthogonal polynomials in \eqref{eq:general3TR}. In the general case, we find that all relevant marix orthogonal polynomials reduce to scalar orthogonal polynomials on the diagonal. 
The symmetric case is relevant, since the recursion \eqref{eq:recursionscalarLorthopols} can be solved in terms of scalar orthogonal polynomials, see 
\cite{LourVA} for the case $L=2$, which can be generalised to the case of general $L$. 
In Section \ref{sec:klr-constant-symmetric} we discuss an example of a symmetric case in more detail for $L=2$. We assume that the constants $a_n$ and $b_{n,2}$ are independent of $n$, 
and that the other constants vanish. The matrix polynomials are diagonal and we can solve the diagonal terms in linear combinations of Chebyshev polynomials. It is remarkable that in this case also a discrete mass point at zero occurs in either 
the spectral decomposition of $JJ^\ast$ or that of $J^\ast J$. 
The more general case of constant coefficients in 
\eqref{eq:recursionscalarLorthopols} can be approached using the techniques of \cite{Roma}, but we will not discuss this here in order to keep the size of the paper within reasonable bounds. 
Note that for $L>1$ the operator $J$ is not normal, but in case of the constant coefficients the operator $J$ is essentially normal, i.e. $JJ^\ast - J^\ast J$ is a compact operator, see 
\cite{Davi} for an overview of these operators. This is clear in the example discussed in Section \ref{sec:klr-constant-symmetric} and is more generally true for the case of constant coefficients. 

We stress that this paper is a study of the operator $J$, and not so much a study of the $L$-orthogonal polynomials $(p_n)_{n\in \N}$ defined by \eqref{eq:recursionscalarLorthopols}. There is a connection between these polynomials and the matrix orthogonal polynomials introduced in Section \ref{sec:MOPrelatedtoJ}. 
This also involves the dual sequence, see \cite{LourVA}, \cite{Maro}. These results will be discussed in a forthcoming paper \cite{KoelLR}. 

It is a natural question, following \cite{PushS}, see also \cite{EichLY}, if one can characterise the operator $J$ in terms of some of the spectral information obtained. Moreover, it would be useful to be able to generalise to the unbounded case.

We have verified statements in the paper  for specific values of $L$ and parameters using Python and Maple, especially for 
Section \ref{sec:symmetriccase} and Section  
\ref{sec:klr-constant-symmetric}.

\medskip
\textbf{Acknowledgements.} 
The research for this paper was started at the HOPE conference organised by Wadim Zudilin and Maarten van Pruijssen, and we thank them for bringing us together. We have benefitted from the Isaac Newton Institute (INI) Retreat programme in Cambridge (EK and AL), and from the 
Research-in-Groups programme at the International Centre for Mathematical Sciences (ICMS) in Edinburgh. We warmly thank INI and ICMS for their excellent hospitality and excellent working conditions. Moreover, we thank the Dutch funding agency NWO under project OCENW.M20.108 for funding a research trip of PR to Nijmegen. 

%%%%%%%%%%%%%%%%%%%%%%%%%%%%%%%%%
%%%%%%%%%%%%%%%%%%%%%%%%%%%%%%%%%
%%%%%%%%%%%%%%%%%%%%%%%%%%%%%%%%%
%%%%%%%%%%%%%%%%%%%%%%%%%%%%%%%%%

\section{Brief recap of matrix orthogonal polynomials}\label{sec:recapMOP}

This section is meant to recall the interaction between matrix measures and 
matrix orthogonal polynomials via the spectral analysis of the matrix three-term recursion. It is also meant to fix some notation. This material is mainly based on the lecture notes by Berg \cite{Berg} and the overview paper by Damanik, Pushnitski and Simon \cite{DamaPS}, see also \cite{DunfS}. 

%%%%%%%%%%%%%%%%%%%%%%%%%%%%%%%%%
%%%%%%%%%%%%%%%%%%%%%%%%%%%%%%%%%
\subsection{Matrix measures}\label{ssec:matrixmeasure} A matrix measure is a $\si$-additive map 
$\mu \colon \cB(\R) \to M_N(\C)$ from the Borel sets on $\R$ to $N\times N$-matrices with 
complex entries denoted by $M_N(\C)$. Moreover, we assume that $\langle \mu(A)x, x\rangle_{\C^N} \geq 0$ 
for all $A\in \cB(\R)$ and all $x\in \C^N$. Then $A\mapsto \Tr(\mu(A))$ gives a 
positive  Borel measure on $\R$, which is denoted $\mu_\Tr$. Picking an 
orthonormal basis $(e_i)_{i=1}^N$ of $\C^N$, the map $A\mapsto 
\mu_{i,j}(A) = \langle \mu(A)e_j,e_i\rangle_{\C^N}$ is a complex Borel measure on $\R$, which is absolutely continuous with respect to $\mu_\Tr$ and we put 
$W_{i,j} = \frac{d\mu_{i,j}}{d\mu_\Tr}$ for the Radon-Nikodym derivative and 
we define the $N\times N$-matrix valued function $W(x) = (W_{i,j}(x))_{i,j=1}^N$. 
Then $W$ is positive definite $\mu_\Tr$-a.e. and moreover  
$0\leq W \leq \Id$ again $\mu_\Tr$-a.e.. 
Then for a measurable $f\colon \R \to \C^N$ we say that $f\in \cL^2(\mu; \C^N)$ if  
\[
\cN(f) = 
\int_\R f^\ast (x) \, d\mu(x) f(x) = 
\int_\R f^\ast (x) \, W(x) \, f(x) \, d\mu_\Tr(x) < \infty.  
\]
For $f,g\in \cL^2(\mu; \C^N)$ we define the sesquilinear form 
$\langle f, g\rangle = \int_\R g^\ast (x) \, d\mu(x) f(x)$. 
Then $L^2(\mu; \C^N)$, the completion of $\cL^2(\mu; \C^N)/ \Ker \cN$, is the corresponding 
Hilbert space. 

%%%%%%%%%%%%%%%%%%%%%%%%%%%%%%%%%
%%%%%%%%%%%%%%%%%%%%%%%%%%%%%%%%%
\subsection{Matrix orthogonal polynomials}\label{ssec:matrixOP}
Next we assume that $x\mapsto x^n W_{i,j}(x)$ is in $\cL^1(\mu_\Tr)$ for all $n\in \N$ and all $i,j\in \{1,\cdots, N\}$, so that 
we have finite moments for the matrix measure $\mu$. In particular, 
a polynomial $q\colon \R \to \C^N$, i.e.  each entry $q_i\in \C[x]$, is 
in $\cL^2(\mu; \C^N)$. We assume that $\cN(q)=0$ if and only if $q$ is the zero polynomial. 
For matrix polynomials $P, Q \in M_N(\C)[x]$, we define a matrix valued inner product by 
\begin{equation}
\langle P, Q\rangle = \int_\R P(x) \, W(x)\, (Q(x))^\ast\, d\mu_\Tr(x)
\end{equation}
so that $\langle P, Q\rangle_{i,j}$ is the inner product of the adjoint of 
the $j$-th column of $Q$ and the adjoint of the  $i$-th column of $P$. Then 
$\langle P,Q\rangle = \langle Q, P\rangle^\ast$, 
$\langle P+R,Q\rangle = \langle P,Q\rangle + \langle R,Q\rangle$,  
$\langle AP,Q\rangle = A \langle P,Q\rangle$ for $P,Q,R\in M_N(\C)[x]$ and 
$A\in M_N(\C)$. Moreover, $\langle P,P\rangle \geq 0$ and 
$\langle P,P\rangle = 0$ if and only if $P=0$, so that the matrix inner product is non-degenerate. 
There exists a sequence of polynomials $(P_n)_{n\in \N}$ with 
degree of $P_n$ equal to $n$ with $P_n(x) = \sum_{k=0}^n A_k(n) x^k$, $A_k(n)\in M_N(\C)$ and 
$A_n(n)$ invertible, so that $\langle P_n, P_m\rangle = \de_{m,n}\Id$. 
Since multiplication by $x$ is symmetric, i.e. $\langle xP,Q\rangle = \langle P,xQ\rangle$, we find that the orthonormal matrix polynomials satisfy a three-term recurrence relation of the form 
\begin{equation}\label{eq:genMOP3termrecurrence}
x\, P_n(x) = A_n P_{n+1}(x) + B_n P_n (x) + A_{n-1}^\ast P_{n-1}(x)
\end{equation}
with the convention that $P_{-1}(x)=0$, where $A_n$ is an invertible and 
$B_n$ is a self-adjoint $N\times N$-matrix for all $n\in \N$. 
We assume the convention that $A_{-1}=0$. Note that with 
the sequence $(P_n)_{n\in \N}$ of orthonormal polynomials also 
the sequence $(U_nP_n)_{n\in \N}$ are orthonormal polynomials for the same measure if $U_n$ is a unitary matrix for all $n\in \N$. 
In particular, $Q_n(x)=U_nP_n(x)$ satisfies 
\begin{equation}\label{eq:genMOP3termrecurrence-unitary}
x\, Q_n(x) = U_n A_n U_{n+1}^\ast  Q_{n+1}(x) + U_n B_n U_n^\ast Q_n (x) + U_n A_{n-1}^\ast U_{n-1}^\ast Q_{n-1}(x). 
\end{equation}
In particular, if the coefficients in 
\eqref{eq:genMOP3termrecurrence} are constant, then this is in general not true for the coefficients in 
\eqref{eq:genMOP3termrecurrence-unitary}. For this reason, matrix orthogonal polynomials are often normalised as monic polynomials, i.e. the leading coefficient is the identity matrix, but this does not suit viewing the corresponding difference operators as 
self-adjoint operators. 

%%%%%%%%%%%%%%%%%%%%%%%%%%%%%%%%%
%%%%%%%%%%%%%%%%%%%%%%%%%%%%%%%%%
\subsection{Difference equations associated to matrix orthogonal polynomials}\label{ssec:differenceeqtnmatrixOP}
Associated to \eqref{eq:genMOP3termrecurrence} we have the 
operator $S$ on the space $\cF(\N;\C^N)$ of functions $F\colon \N \to \C^N$ given 
by 
\begin{equation}\label{eq:genMOPrecurrenceoperator}
\bigl( S F\bigr)(n) = A_n F(n+1) + B_n F(n) + A_{n-1}^\ast F(n-1)
\end{equation}
with $F(-k)=0$ for $k\geq 1$. By comparing \eqref{eq:genMOPrecurrenceoperator} with \eqref{eq:genMOP3termrecurrence} we see that $F(n) = P_n(z) \xi$ for 
fixed $\xi\in \C^N$ is a solution to $SF=zF$. Assuming that the operator 
norms of $A_n$ and $B_n$ are bounded independent of $n\in \N$, we can view 
$S \colon \ell^2(\N;\C^N)\to \ell^2(\N;\C^N)$ as a bounded self-adjoint operator. Here $\ell^2(\N;\C^N)$ is the space of square
summable $\C^N$-valued functions on $\N$, i.e.  
$\ell^2(\N;\C^N) =\{ F\colon \N \to \C^N\mid 
\sum_{n\in \N} \langle F(n), F(n)\rangle_{\C^N}<\infty\}$. 
Define 
$F_{k,\xi} \colon \N \to \C^N$, $F_{k,\xi}(n) =\de_{k,n}\xi$, for $\xi\in \C^N$.
Picking an orthonormal basis $(e_i)_{i=1}^N$ of $\C^N$, then $\{ F_{0,e_i} \mid 
1\leq i \leq N\}$ is a cyclic set for $S$, i.e. the 
closure of the linear span of elements of $S^p F_{0,e_i}$ for 
$p\in \N$ and $1\leq i\leq N$ gives the Hilbert space $\ell^2(\N;\C^N)$. Indeed, this follows from 
\begin{equation}\label{eq:genMOPrecurrenceoperatoronFkxi}
S F_{k,\xi} = F_{k-1, A_{k-1}\xi} + F_{k,B_k\xi} +
F_{k+1, A_{k}^\ast\xi},
\end{equation}
the invertibility of $A_n$ for $n\in \N$ and that  
$\{ F_{k,e_i} \mid k\in \N, 1\leq i \leq N\}$ forms an orthonormal basis for $\ell^2(\N;\C^N)$. More precisely, we can 
write any element $F_{k,\xi} = \sum_{p\in \N} \sum_{i=1}^N 
\al_{p,i} S^p F_{0,e_i}$ as a finite sum where $\al_{p,i}\in \C$. 
By the spectral theorem 
there exists a matrix measure $\mu$ and a unitary map 
$\cF \colon \ell^2(\N;\C^N) \to L^2(\mu; \C^N)$ intertwining $S$ with the multiplication 
operator. Then $\cF$ is given by $\bigl( \cF F_{k,\xi}\bigr)(x) 
= Q_k(x)\xi$, where $Q_k(z) = P_k^\ast(z) = \bigl( P_k(\overline{z})\bigr)^\ast$, where the last $\ast$ denotes the adjoint of a matrix.  

Conversely, given a 
self-adjoint operator $S$ of the form \eqref{eq:genMOPrecurrenceoperator} for uniformly bounded matrices $A_n$ and $B_n$, there exists a non-degenerate inner
product which can be written in terms of a matrix measure. This is the analogue of Favard's theorem in this case. In order to descibe this, we let 
$P_n^{(1)}(x)$ be the associated matrix polynomials to \eqref{eq:genMOP3termrecurrence}, i.e. satisfying \eqref{eq:genMOP3termrecurrence} for $n\geq 1$ with initial condition 
$P_0^{(1)}(x)=0$ and $P_1^{(1)}(x)=A_0^{-1}$. In particular, the associated 
polynomials satisfy 
\begin{equation}\label{eq:genMOP3associatedtermrecurrence}
\Id\de_{n,0} + x\, P^{(1)}_n(x) = A_n P^{(1)}_{n+1}(x) + B_n P^{(1)}_n (x) + A_{n-1}^\ast P^{(1)}_{n-1}(x), \qquad n\in \N. 
\end{equation}
Then there exists a asymptotically free solution $F_z(n) = 
P^{(1)}_n(z) + P_n(z)M(z)$, such that $n\mapsto F_z(n)\xi$ in $\ell^2(\N;\C^N)$ satisfies  
$\bigl( SF_z\bigr) (n) = zF_z(n)$ for $n\geq 1$. Then the resolvent $R_z=(S-z)^{-1}$, $z\not\in \si(S)$,  for 
$S$ can be obtained in terms of $P_n(z)$ and $F_n(z)$ using the Liouville-Ostrogradskii identities. Then $M\colon \C\setminus \R \to M_N(\C)$ is the Stieltjes transform of the matrix measure:
\begin{equation}
M(z) = \int_\R \frac{1}{x-z} \, d\mu(x), \qquad z\in \C\setminus \R. 
\end{equation}
This coincides with the upper left $N\times N$-block of the resolvent. 
By the Stieltjes-Perron inversion formula one obtains the matrix measure 
$\mu$ from its Stieltjes transform $M$, which in turn follows from the 
asymptotically free solution.

%%%%%%%%%%%%%%%%%%%%%%%%%%%%%%%%%
%%%%%%%%%%%%%%%%%%%%%%%%%%%%%%%%%
%%%%%%%%%%%%%%%%%%%%%%%%%%%%%%%%%
%%%%%%%%%%%%%%%%%%%%%%%%%%%%%%%%%

\section{Difference operators}\label{sec:differenceoperators}

In this section we consider the difference operators 
$J$ and $J^\ast$ as defined in \eqref{eq:defoperatorJassoctopols} 
and \eqref{eq:defoperatorJstarassoctopols}. 
The difference operator is rewritten in terms of a second order matrix recurrence operator, and we introduce a self-adjoint operator $\J$ to $J$ and $J^\ast$ as in \cite[Thm.~7.3.7]{HornJ}, 
\cite{PushS}. 

Associated to $J$ and $J^\ast$ we introduce 
$\J \colon \cF(\N) \oplus \cF(\N) \to \cF(\N) \oplus \cF(\N)$ defined by 
\begin{equation}
\J  = \begin{pmatrix} 0 & J \\ J^\ast & 0 \end{pmatrix}, 
\quad 
\Bigl( \J \begin{pmatrix} f \\ g \end{pmatrix} \Bigr) (n) =
\begin{pmatrix} a_n g(n+1) + \sum_{k=0}^L b_{n,k} g(n-k) \\ 
\sum_{k=0}^L \overline{b_{n+k,k}} f(n+k)
+ a_{n-1} f(n-1)\end{pmatrix}.
\end{equation}
Since we assume bounded coefficients, $\J \colon \ell^2(\N)\oplus \ell^2(\N)\to \ell^2(\N) \oplus \ell^2(\N)$ is a bounded self-adjoint operator. 

We want to rewrite the operators in terms of operators on $\cF(\N;\C^N)$ for suitable $N$. Note that $\J^2$ involves the operators $JJ^\ast$ and $J^\ast J$, which are self-adjoint operators on $\ell^2(\N)$. 
Since $JJ^\ast$ and $J^\ast J$ act as 
$2L+3$-recurrence operators, we want to identify with 
$(L+1)\times(L+1)$-matrix orthonormal polynomials following \cite{DuraVA}. This is done in Section \ref{sec:MOPrelatedtoJ}. 
In order to write $J$, $J^\ast$ and $\J$ as recursion operators 
on $\cF(\N;\C^{L+1})$, we use Lemma \ref{lem:fromTonFNtoTonFNCN} whose proof is straightforward, see \cite{DuraVA} for the case of a self-adjoint operator $T$.

\begin{lemma}\label{lem:fromTonFNtoTonFNCN}
Let $T\colon \cF(\N) \to \cF(\N)$ be an operator 
of the form 
$(T f)(n) = \sum_{k=-N}^N \al^n_{k}\, f(n+k)$
with the convention $\al^n_k=0$ if $n+k<0$. 
The spaces $\cF(\N)$ and $\cF(\N; \C^N)$ are identified via 
\begin{align*}
\iota \colon \cF(\N) \to \cF(\N; \C^N), 
\qquad F(n) =  \iota(f)(n) = 
\begin{pmatrix} f(nN) \\f(nN+1) \\ \vdots \\ f(nN+N-_i1) \end{pmatrix}
= \sum_{i=0}^{N-1} f(nN+i) v_i,
\end{align*}
where $(v_i)_{i=0}^{N-1}$ is the standard orthonormal basis of $\C^N$.
In particular, 
$\langle \iota(f)(n), v_r\rangle_{\C^N} = f(nN+r)$ and 
$(\iota^{-1} F)(p) = \langle F(n), v_r\rangle_{\C^N}$, for $p=nN+ r$, $0\leq r <N$. 
Then  
$$((\iota \circ T \circ \iota^{-1}) F)(p) = 
A_p F(p+1) + B_p F(p) + C_p F(p-1)$$ for $F\in \cF(\N;\C^N)$
with $A_p$ a lower triangular $N\times N$-matrix, $C_p$ an upper triangular $N\times N$-matrix and 
$B_p$ a $N\times N$-matrix given for $p\in \N$ by 
\begin{align*}
(A_p)_{i,j} = \begin{cases}
0, & j> i,  \\
\al^{pN+i}_{j-i+N}, & j\leq i, 
\end{cases}, \qquad 
(C_p)_{i,j} = \begin{cases}
0, & j< i,  \\
\al^{pN+i}_{j-i-N}, & j\geq i, 
\end{cases}\qquad 
(B_p)_{i,j} = \al^{pN+i}_{j-i}
\end{align*}
with the convention $C_0=0$, which is in line with the convention 
$\al^n_k=0$ for $n+k<0$. 
Here the matrix entries are with respect to the standard 
orthonormal base $(v_i)_{i=0}^{N-1}$ of $\C_N$, so 
$(A_p)_{i,j} = \langle A_p v_j, v_i\rangle_{\C^N}$, etc. 
\end{lemma}

So we apply Lemma \ref{lem:fromTonFNtoTonFNCN} with $N=L+1$ to $J$ and $J^\ast$. Putting 
$J' = \iota \circ J \circ \iota^{-1} \colon \cF(\N;\C^{L+1}) \to \cF(\N;\C^{L+1})$ 
we get the following explicit expression. 

\begin{cor}\label{cor:lem:fromTonFNtoTonFNCN}
$J' \colon \cF(\N;\C^{L+1}) \to \cF(\N;\C^{L+1})$ is given by 
$$\bigl( J' F\bigr) (p) = X_p F(p+1) + Y_p F(p) + Z_p F(p-1)$$ with
$(X_p)_{i,j} = \de_{i,L}\de_{j,0} a_{p(L+1)+L}$, 
$(Z_p)_{i,j} = b_{p(L+1)+i, L+1+i-j}$ for $j>i$ and 
$(Z_p)_{i,j} = 0$ for $j\leq i$, 
$(Y_p)_{i,j} = b_{p(L+1)+i, i-j}$ for $i\geq j$,  
$(Y_p)_{i,i+1} = a_{p(L+1)+i}$ for $0\leq i <L$, and $(Y_p)_{i,i+1} =0$ for $j>i+1$.
In general the indices range in $0\leq i, j \leq L$.
\end{cor}

As $(L+1)\times (L+1)$-matrices, the matrices $X_p$, $Z_p$ are singular. Note that 
the one non-trivial entry of $X_p$ is non-zero, and that the elements on the superdiagonal of $Z_p$ are all non-zero. Moreover, the elements on the superdiagonal of $Y_p$ are all non-zero as well as the matrix entry $(Y_p)_{N,0}$. 
Also, $Z_0=0$ by the convention $b_{0,k}=0$ for $1\leq k \leq L$. 
Explicitly, the matrices $X_p$, $Y_p$ and $Z_p$ are 
\begin{equation}\label{eq:defXpYpZp}
\begin{split}
X_p &\, =
\begin{pmatrix}
0     & 0      & \cdots & \cdots    & 0 \\
0     & 0      & \cdots & \cdots    & 0 \\
\vdots&        &\ddots  &           & \vdots \\
\vdots&        &        &\ddots     & \vdots \\
a_{p(L+1)+L}& 0& \cdots &            &  0
\end{pmatrix}, \quad
Z_p =
\begin{pmatrix}
0     & b_{p(L+1),L}      & b_{p(L+1),L-1} & \cdots    & b_{p(L+1),1} \\
0     & 0      & b_{p(L+1)+1,L}  & \cdots    & b_{p(L+1)+1,2} \\
\vdots&        &\ddots  &  \ddots          & \vdots \\
0&   0    &  \cdots    & 0    &  b_{p(L+1)+L-1,L}\\
0     &       0& \cdots &            &  0
\end{pmatrix}, 
\\
Y_p &\,=
\begin{pmatrix}
b_{p(L+1),0}     & a_{p(L+1)}          & 0      & \cdots    & 0 \\
b_{p(L+1)+1,1}     & b_{p(L+1)+1,0}      & a_{p(L+1)+1} & 0         & 0 \\
\vdots&        &\ddots  &  \ddots          & \vdots \\
b_{p(L+1)+L-1,L-1}& b_{p(L+1)+L-1,L-2}       &  \cdots       &b_{p(L+1)+L-1,0}     & a_{p(L+1)+L-1} \\
b_{p(L+1)+L,L} & b_{p(L+1)+L,L-1}     & \cdots   &  b_{p(L+1)+L,1}   &  b_{p(L+1)+L,0}
\end{pmatrix}.
\end{split}
\end{equation}

Similarly, we get an expression for the formal adjoint of $J^\ast$ by applying 
Lemma \ref{lem:fromTonFNtoTonFNCN} to \eqref{eq:defoperatorJstarassoctopols} or by taking the formal adjoint of the 
result in Corollary \ref{cor:lem:fromTonFNtoTonFNCN}. 

\begin{cor}\label{cor:lem:fromTonFNtoTonFNCN2}
$(J')^\ast = \iota \circ J^\ast \circ \iota^{-1} \colon \cF(\N;\C^{L+1}) \to \cF(\N;\C^{L+1})$ is given by 
$$\bigl( (J')^\ast F\bigr) (p) = Z_{p+1}^\ast F(p+1) + Y_p^\ast F(p) + X_{p-1}^\ast F(p-1)$$ 
with $X_p$, $Y_p$ and $Z_p$ as in 
Corollary \ref{cor:lem:fromTonFNtoTonFNCN}.
\end{cor}

Note that $X_{-1}=0$, which follows from the convention $a_{-k}=0$ for $k\geq 1$. 

In Corollary \ref{cor:lem:fromTonFNtoTonFNCN} and Corollary \ref{cor:lem:fromTonFNtoTonFNCN2} we have indicated the action of $J$ and $J^\ast$ on 
$\cF(\N;\C^{L+1})$ by a prime, but from now on we identify these operators as it will be clear from the context on which spaces they act. 

Identifying $\cF(\N;\C^{L+1})\oplus \cF(\N;\C^{L+1}) \cong \cF(\N; \C^{L+1}\oplus \C^{L+1})$ 
we get $\J \colon \cF(\N; \C^{L+1}\oplus \C^{L+1}) \to \cF(\N; \C^{L+1}\oplus \C^{L+1})$ in the matrix representation 
\begin{equation}\label{eq:matrixreprbbJ}
\begin{split}
\bigl( \J F\bigr)(p) &\, = 
T_p F(p+1) + S_p F(p) + T_{p-1}^\ast F(p-1), \\
T_p &\, =\begin{pmatrix} 0 & X_p \\ Z_{p+1}^\ast & 0 \end{pmatrix}, 
\qquad S_p =\begin{pmatrix} 0 & Y_p \\ Y_{p}^\ast & 0 \end{pmatrix}. 
\end{split}
\end{equation}
for $F\in \cF(\N; \C^{L+1}\oplus \C^{L+1})$. 
We use the orthonormal basis for $\C^{L+1}\oplus \C^{L+1}$ given by 
$(v^1_i)_{i=0}^L$ for the first copy of $\C^{L+1}$ and $(v^2_i)_{i=0}^L$ for the second copy. Note that $T_p \in M_{2(L+1)}(\C)$ is a singular matrix, and we want to reduce the expression for $\J$ in \eqref{eq:matrixreprbbJ} to suitable subspaces. For this we need to analyse $T_p$, and this is done in Lemma \ref{lem:rangekernelT}. 

\begin{lemma}\label{lem:rangekernelT}
Let $V_1 = \C v^1_L \oplus \bigoplus_{i=1}^L \C v^2_i$, 
$V_2= \bigoplus_{i=0}^{L-1} v_i^1 \oplus \C v_0^2$ be $L+1$-dimensional subspaces of 
$\C^{L+1}\oplus \C^{L+1}$. Then $V_1^\perp= V_2$, $\C^{L+1}\oplus \C^{L+1} = V_1\oplus V_2$ and 
$\Ker T_p = \Ran T_p^\ast = V_1$, $\Ker T_p^\ast = \Ran T_p = V_2$ for all $p\in \N$.
Define the linear operator $\Om\colon \C^{L+1}\oplus \C^{L+1} \to \C^{L+1}\oplus \C^{L+1}$ by $\Om v_i^1 =v_i^1$, $0\leq i<L$, $\Om v_L^1 =v_0^2$, 
$\Om v_0^2 = v_L^1$, $\Om v_i^2 =v_i^2$, $1\leq i\leq L$, then $\Om^\ast = \Om = \Om^{-1}$ and 
\begin{equation*}
\Om T_p \Om = \begin{pmatrix} 0 & 0 \\ T^{21}_p & 0 \end{pmatrix}, 
\qquad \Om S_p \Om = \begin{pmatrix} S^{11}_p & S^{12}_p \\ S^{21}_p & S^{22}_p \end{pmatrix}
\end{equation*}
with the $(L+1)\times (L+1)$-matrices $T^{21}_p$, $S^{11}_p$, $S^{12}_p$, $S^{21}_p$ and $S^{22}_p$ satisfying $(S^{11}_p)^\ast = S^{11}_p$, $(S^{22}_p)^\ast = S^{22}_p$, 
$(S^{12}_p)^\ast = S^{21}_p$ and $T^{21}_p$, $S^{21}_p$, $S^{12}_p$ are invertible for all $p\in \N$. 
\end{lemma}

\begin{proof} Most statements follow directly from the construction. The first statements follow from \eqref{eq:matrixreprbbJ}, $S_p$ being self-adjoint 
and Corollary \ref{cor:lem:fromTonFNtoTonFNCN} or using 
\eqref{eq:defXpYpZp} and the fact that $a_n\not=0$ and $b_{n,L}\not=0$ for $n\geq L$. 
Observe that conjugation with $\Om$ affects only the two middle rows and columns of
the matrix in $M_{2(L+1)}(\C)$, which are both flipped. So we find  
\begin{equation*}
T^{21}_p =
\begin{pmatrix}
0     &  0     &   \cdots   & 0  &  a_{p(L+1)+L} \\
\overline{b_{(p+1)(L+1),L}}     & 0      & 0  & \cdots    &0 \\
\vdots&  \ddots       &\ddots  &           & \vdots \\
\overline{b_{(p+1)(L+1),2}} &   \overline{b_{(p+1)(L+1)+1,3}}    &  \cdots    & 0    &  0\\
\overline{b_{(p+1)(L+1),1}} &  \overline{b_{(p+1)(L+1)+1,2}}      & \cdots &   \overline{b_{(p+1)(L+1)+L-1,L}}         &  0
\end{pmatrix}
\end{equation*}
and since the subdiagonal has non-zero entries and $(T_p^{21})_{0,L}\not=0$ we see that 
$\det(T^{21}_p )\not=0$. And similarly,
\begin{equation*}
S^{12}_p =
\begin{pmatrix}
0    & a_{p(L+1)}          & 0      & \cdots    & 0 \\
0     & b_{p(L+1)+1,0}      & a_{p(L+1)+1} & 0         & 0 \\
\vdots&        &\ddots  &  \ddots          & \vdots \\
0 & b_{p(L+1)+L-1,L-2}       &  \cdots       &b_{p(L+1)+L-1,0}     & a_{p(L+1)+L-1} \\
\overline{b_{p(L+1)+L,L}} & 0     & \cdots   & 0   &  0 
\end{pmatrix}
\end{equation*}
and $\det (S^{12}_p) \not=0$ follows. Since $S^{21}_p = (S^{12}_p)^\ast$, both are invertible.  
\end{proof}

Note that $T^{21}_p= X_p^\ast+Z_{p+1}^\ast$, but there is no similar description of $S^{12}_p$ in general. 
For completeness, we state the form of the $(L+1)\times (L+1)$-matrices $S^{11}_p$ and $S^{22}_p$:
\begin{equation}\label{eq:defS11andS22}
\begin{split}
S^{11}_p &\, = 
\begin{pmatrix}
0     &   0    & \cdots & 0     & b_{p(L+1),0} \\
0     &  \ddots     &  &  0   & b_{p(L+1)+1,1} \\
\vdots&        &\ddots  &           & \vdots \\
0 & 0       &  \cdots       & 0     & b_{p(L+1)+L-1,L-1} \\
\overline{b_{p(L+1),0}}& \overline{b_{p(L+1)+1,1}} & \cdots & \overline{b_{p(L+1)+L-1,L-1}}         &  0
\end{pmatrix}, \\
S^{22}_p &\, = 
\begin{pmatrix}
0     &   b_{p(L+1)+L, L-1}    & \cdots & b_{p(L+1)+L,1}     & b_{p(L+1)+L,0} \\
\overline{b_{p(L+1)+L, L-1}}     &  0     &  \cdots &   0  &  0 \\
\vdots&        &\ddots  &           & \vdots \\
\overline{b_{p(L+1)+L,1}} & 0       &  \cdots       &   0  & 0 \\
\overline{b_{p(L+1)+L,0}}& 0 & \cdots &       0   &  0
\end{pmatrix}.
\end{split}
\end{equation}

Note that $S^{11}_p$ and $S^{22}_p$ are singular matrices, and that $S^{11}_p=0=S^{22}_p$ as matrices exactly when 
$b_{n,k}=0$ for all $0\leq k<L$. This is called the symmetric case of \eqref{eq:recursionscalarLorthopols}, 
and this is discussed in more detail in 
Section \ref{sec:symmetriccase}.

%%%%%%%%%%%%%%%%%%%%%%%%%%%%%%%%%
%%%%%%%%%%%%%%%%%%%%%%%%%%%%%%%%%
%%%%%%%%%%%%%%%%%%%%%%%%%%%%%%%%%
%%%%%%%%%%%%%%%%%%%%%%%%%%%%%%%%%

\section{Matrix orthogonal polynomials related to 
\texorpdfstring{$J$}{J}}\label{sec:MOPrelatedtoJ}

We have noted that $JJ^\ast$ and $J^\ast J$ occur in $\J^2$, and that these operators are self-adjoint $2L+3$-difference operators. 
According to Dur\'an and Van Assche \cite{DuraVA}, they are related to 
$(L+1)\times (L+1)$-matrix orthogonal polynomials. Before introducing these polynomials, we consider the polynomials 
associated with $\J$ using Lemma \ref{lem:rangekernelT}. 

The map $\Om \colon \C^{L+1}\oplus\C^{L+1} \to \C^{L+1}\oplus\C^{L+1}$ of Lemma \ref{lem:rangekernelT} 
gives a map $$\Om \colon \cF(\N;\C^{L+1}\oplus\C^{L+1}) \to 
\cF(\N;\C^{L+1}\oplus\C^{L+1})$$ by acting termwise: $(\Om F)(p) = \Om F(p)$.
Then Lemma \ref{lem:rangekernelT} gives 
\begin{equation}\label{eq:actionOmegabbJOmega}
\bigl( (\Om \J \Om) F\bigr)(p) = 
\begin{pmatrix} 0 & 0 \\ T^{21}_p & 0 \end{pmatrix} F(p+1)
+ \begin{pmatrix} S^{11}_p & S^{12}_p \\ S^{21}_p & S^{22}_p \end{pmatrix} F(p)
+ \begin{pmatrix} 0 & (T^{21}_{p-1})^\ast \\ 0 & 0 \end{pmatrix} F(p-1)
\end{equation}
for $F\in \cF(N;\C^{L+1}\oplus\C^{L+1})$. 

Now $\Om \J\Om$ gives rise to a bounded self-adjoint operator on 
$\ell^2(\N;\C^{L+1}\oplus\C^{L+1})$, and we determine the subspaces generating $\ell^2(\N;\C^{L+1}\oplus\C^{L+1})$.
Let $F_{k, \xi\oplus\eta}(n) = \de_{k,n} \xi\oplus \eta$ as 
in Section \ref{ssec:differenceeqtnmatrixOP} for $k\in \N$ and $\xi,\eta\in \C^{L+1}$.

\begin{lemma}\label{lem:cyclicvectorsOmegabbJOmega} 
Any element $F_{k,\xi\oplus \eta} = \sum_{p\in \N}  
\al_{p,i} (\Om \J\Om)^p F_{0,\xi_p\oplus 0}$ as a finite sum where $\al_{p}\in \C$ and $\xi_p\in \C^{L+1}$. Consequently, the 
closure of the linear span of elements of the form 
$(\Om \J\Om)^p F_{0, \xi\oplus 0}$ for $p\in \N$, $\xi \in \C^{L+1}$ is  $\ell^2(\N;\C^{L+1}\oplus\C^{L+1})$. 
\end{lemma}

Lemma \ref{lem:cyclicvectorsOmegabbJOmega} implies that the 
multiplicity of an element of the spectrum of the self-adjoint operator $\Om\J\Om$ is at most $L+1$. 

\begin{proof} As in Section \ref{ssec:differenceeqtnmatrixOP} we note that 
\begin{equation}\label{eq:lem:cyclicvectorsOmegabbJOmega}
(\Om \J \Om) F_{k,\xi\oplus \eta} = 
F_{k-1,0\oplus T^{21}_{k-1}\xi} + 
F_{k, (S^{11}_k\xi +S^{12}_k\eta)\oplus( S^{21}_k \xi + S^{22}_k\eta)} 
+ F_{k+1, (T^{21}_k)^\ast \eta \oplus 0}
\end{equation}
and in particular
\begin{equation*}
(\Om \J \Om) F_{0,\xi\oplus 0} = 
F_{0, S^{11}_0\xi\oplus  S^{21}_0 \xi }.
\end{equation*}
Lemma \ref{lem:rangekernelT} gives that $S^{21}_0$ is invertible, so that 
\[
F_{0, 0\oplus\eta} = (\Om \J\Om)(F_{0, (S^{21}_0)^{-1}\eta\oplus 0}) 
- F_{0, S_0^{11}(S^{21}_0)^{-1}\eta\oplus 0}.
\]
So we have proved the $k=0$ case of the statement that 
any $F_{l, \xi\oplus \eta}$ for $l\leq k$ can be written as 
a finite sum of the form $\sum_{p\in \N}  
\al_{p,i} (\Om \J\Om)^p F_{0,\xi_p\oplus 0}$. Next we use induction on $k$, and applying 
\eqref{eq:lem:cyclicvectorsOmegabbJOmega} and the invertibility of $T^{21}_k$ by Lemma \ref{lem:fromTonFNtoTonFNCN}, we see that 
$F_{k+1, \xi\oplus 0}$ can be written as such a finite sum. As for the initial step we now observe that, again using Lemma \ref{lem:fromTonFNtoTonFNCN},
\begin{equation*}
F_{k+1,0\oplus \eta} = 
(\Om \J \Om)(F_{k+1, (S^{21}_{k+1})^{-1}\eta \oplus 0}) 
- F_{k+1, S^{11}_{k+1}(S^{21}_{k+1})^{-1}\eta \oplus 0} 
- F_{k, 0 \oplus T^{21}_k (S^{21}_{k+1})^{-1}\eta}
\end{equation*}
so that the induction step follows. 
\end{proof}

In particular, putting $F(p) = \begin{pmatrix} G(2p) \\ G(2p+1) \end{pmatrix}$ 
for $G\in \cF(\N; \C^{L+1})$ yields 
\begin{equation}\label{eq:recurrenceOmbbJOm}
\begin{split}
\bigl( (\Om \J \Om) F\bigr)(p) = 
\begin{pmatrix}
S^{11}_p G(2p) + S^{12}_p G(2p+1) + (T^{21}_{p-1})^\ast G(2p-1) \\
T^{21}_p G(2p+2) + S^{21}_p G(2p) +S^{22}_pG(2p+1)
\end{pmatrix}.
\end{split}
\end{equation}
For the corresponding eigenvalue equation we generate 
orthonormal $(L+1)\times (L+1)$-matrix polynomials $P_n$ by  
\begin{equation}\label{eq:MOPPn}
\begin{split}
z\-P_n(z) &\, = A_n \, P_{n+1}(z) + B_n \, P_n(z) + A_{n-1}^\ast P_{n-1}(z), \\
A_n &\, = \begin{cases} S^{12}_{n/2}, & n \text{ even}, \\ T^{21}_{(n-1)/2}, & n \text{ odd},\end{cases}\qquad 
B_n  = \begin{cases} S^{11}_{n/2}, & n \text{ even}, \\ S^{22}_{(n-1)/2}, & n \text{ odd}.\end{cases}
\end{split}
\end{equation}
In order to see that \eqref{eq:MOPPn} gives orthonormal matrix polynomials matching \eqref{eq:recurrenceOmbbJOm}, we need invertibility of $A_n$, self-adjointness of $B_n$ and 
$A_{n-1}^\ast = (T^{21}_{n/2-1})^\ast$ for $n$ even and 
$A_{n-1}^\ast = S^{21}_{(n-1)/2}$ for $n$ odd. This follows directly from 
Lemma \ref{lem:rangekernelT}. 
Note that the convention gives $A_{-1}=T^{21}_{-1}=0$, so that the polynomials
are determined by the initial condition $P_0(z) = \Id$. 
By Lemma \ref{lem:cyclicvectorsOmegabbJOmega} and Section 
\ref{ssec:differenceeqtnmatrixOP} it follows that 
all eigenfunctions of $\Om \J \Om$ are of this form.  
This gives the first family of matrix polynomials associated to $J$. 
Note that $G(p) = P_p(z)\xi$ for $\xi \in \C^{L+1}$ leads to eigenfunctions of 
$\Om \J \Om$. 
By Lemma \ref{lem:cyclicvectorsOmegabbJOmega} and Section 
\ref{ssec:differenceeqtnmatrixOP} it follows that 
all eigenfunctions of $\Om \J \Om$ are of this form.

The other families of matrix polynomials are associated to $JJ^\ast$ and $J^\ast J$. Combining 
Corollary \ref{cor:lem:fromTonFNtoTonFNCN} and 
Corollary \ref{cor:lem:fromTonFNtoTonFNCN2} we see that for $F\in \cF(\N;\C^{L+1})$ 
\begin{equation}\label{eq:JJastonF}
\begin{split}
&\, \bigl( JJ^\ast F\bigr) (p) = A'_p F(p+1) + B'_pF(p) + (A'_{p-1})^\ast F(p-1), \\
&\,A'_p  = X_p  Y_{p+1}^\ast + Y_p Z_{p+1}^\ast, \qquad
B'_p = X_pX_p^\ast + Y_p Y_p^\ast + Z_pZ_p^\ast 
\end{split}
\end{equation}
and 
\begin{equation}\label{eq:JastJonF}
\begin{split}
&\, \bigl( J^\ast J F\bigr) (p) = A''_p F(p+1) + B''_pF(p) + (A''_{p-1})^\ast F(p-1), \\
&\, A''_p  = Z_{p+1}^\ast Y_{p+1} + Y_p^\ast X_p, \qquad
B''_p =  X_{p-1}^\ast X_{p-1} + Y_p^\ast Y_p + Z_{p+1}^\ast Z_{p+1}
\end{split}
\end{equation}
Note that $X_p Z_k^\ast =0 = Z_k^\ast X_p$ for arbitrary $p,k\in\N$ is used in obtaining \eqref{eq:JJastonF} and \eqref{eq:JastJonF}. 
Naturally, \eqref{eq:JJastonF} and \eqref{eq:JastJonF} can also be obtained by
first calculating $JJ^\ast$ and $J^\ast J$ on $\cF(\N)$ and next using Lemma 
\ref{lem:fromTonFNtoTonFNCN} or Dur\'an and Van Assche \cite{DuraVA}.  
In order to associate matrix polynomials to the recurrence relations  
in \eqref{eq:JJastonF} and \eqref{eq:JastJonF} we need 
$A'_p$ and $A''_p$ invertible and $B'_p$ and $B''_p$ to be self-adjoint. 
The last requirement follows by inspection, and the first follows from 
Lemma \ref{lem:ApaccentApdoubleaccentinvertible}.

\begin{lemma}\label{lem:ApaccentApdoubleaccentinvertible}
$A'_p$ and $A''_p$ are lower triangular matrices and 
\begin{equation*}
\begin{split}
(A'_p)_{i,i} &\, = a_{p(L+1)+i} \overline{b_{(p+1)(L+1)+i,L}} \not= 0, \qquad 0\leq i \leq L,  \\
(A''_p)_{i,i} &\, = 
\begin{cases} a_{p(L+1)+L} \overline{b_{p(L+1)+L,L}} \not= 0, & i=0, \\
a_{(p+1)(L+1)+i-1} \overline{b_{(p+1)(L+1)+i-1,L}} \not= 0, & 0 < i \leq L.
\end{cases}
\end{split}
\end{equation*}
\end{lemma}

\begin{proof} Note that 
$Y_p Z_{p+1}^\ast$ is lower triangular using \eqref{eq:defXpYpZp}, with  
the entries
$$
(Y_p Z_{p+1}^\ast)_{i,i} = a_{p(L+1)+i} \overline{b_{(p+1)(L+1)+i,L}} \not= 0, \quad 0\leq i <L, \qquad (Y_p Z_{p+1}^\ast)_{L,L} = 0.
$$
Note that $(X_p Y_{p+1}^\ast)_{L,L} = a_{p(L+1)+L} \overline{b_{(p+1)(L+1)+L, L}} \not= 0$ and that, using \eqref{eq:defXpYpZp},  $X_p Y_{p+1}^\ast$ has only a non-zero last row. This proves the first statement. 

The statement for $A''_p$ follows analogously. 
\end{proof}

Now $JJ^\ast$ and $J^\ast J$ are bounded self-adjoint operators as discussed in Section \ref{ssec:differenceeqtnmatrixOP}. In particular, the multiplicity of the eigenfunctions in 
$\cF(\N;\C^{L+1})$ is $L+1$. 

So we can define two sets of orthonormal matrix polynomials by the recurrences in 
\eqref{eq:JJastonF} and \eqref{eq:JastJonF}.  We define the 
$(L+1)\times (L+1)$-matrix polynomials $\Phi_n$ and $\Psi_n$ by 
\begin{equation}\label{eq:defPhiPsiMOP}
\begin{split}
z \, \Phi_n(z) &\, = A'_n\, \Phi_{n+1}(z) + B'_n\, \Phi_n(z) + (A'_{n-1})^\ast \Phi_{n-1}(z), \\
z \, \Psi_n(z) &\, = A''_n\, \Psi_{n+1}(z) + B''_n\, \Psi_n(z) + (A''_{n-1})^\ast \Psi_{n-1}(z)
\end{split}
\end{equation}
By convention on the coefficients it follows that $A'_{-1}=A''_{-1}=0$, so that the polynomials $\Phi_n$ and $\Psi_n$ are determined by the initial condition 
$\Phi_0(z) = \Id = \Psi_0(z)$. 

Next we connect the $(L+1)\times (L+1)$-matrix orthogonal polynomials 
$P_n$, $\Phi_n$ and $\Psi_n$. From 
\eqref{eq:recurrenceOmbbJOm}  and \eqref{eq:MOPPn} it follows that for 
vectors $\xi\in \C^{L+1}$ the element
\begin{equation}\label{eq:matchPPhiPsi1}
F_w^{\xi} \colon \N \to \C^{L+1}\oplus \C^{L+1}, 
\qquad F_w^{\xi}(p) = \begin{pmatrix} P_{2p}(w)\, \xi \\
                           P_{2p+1}(w)\, \xi 
                           \end{pmatrix}
\end{equation}
satisfies $\Om\J \Om \, F_w^{\xi} = w\, F_w^{\xi}$ and 
thus $\Om\J^2 \Om \, F_w^{\xi} = w^2\, F_w^{\xi}$. On the other hand, 
by \eqref{eq:JJastonF}, \eqref{eq:JastJonF} and \eqref{eq:defPhiPsiMOP} we see that 
for $\si,\tau \in \C^{L+1}$ we have
\begin{equation}\label{eq:matchPPhiPsi2}
\Bigl( \J^2 \begin{pmatrix} \Phi_\bullet(z)\, \si\\ \Psi_\bullet(z)\, \tau \end{pmatrix}\Bigr) (n)  =
\Bigl( \begin{pmatrix} JJ^\ast & 0  \\ 0 & J^\ast J \end{pmatrix}
\begin{pmatrix} \Phi_\bullet (z)\, \si\\ \Psi_\bullet (z)\, \tau \end{pmatrix} \Bigr) (n)  
= z\, \begin{pmatrix} \Phi_n(z)\, \si\\ \Psi_n(z)\, \tau \end{pmatrix}
\end{equation}
using the notation $\Phi_\bullet(z)\si \in \cF(\N;\C^{L+1})$ 
defined by $n\mapsto \Phi_n(z)\si$. 
Note that the eigenvalue equation $\J^2 F = zF$ has $2(L+1)$-degrees of freedom, so that the solutions of \eqref{eq:matchPPhiPsi1} and \eqref{eq:matchPPhiPsi2} 
are related.

\begin{prop}\label{prop:relationPnPhinPsin}  We have for all $p\in \N$ 
\begin{equation*}
\begin{pmatrix} P_{2p}(w)\, \xi \\
                           P_{2p+1}(w)\, \xi 
                           \end{pmatrix}
= \Om 
\begin{pmatrix} \Phi_p(w^2) & 0 \\ 0 & \Psi_p(w^2) \end{pmatrix} \Om 
\begin{pmatrix}  \xi \\ (S^{12}_0)^{-1} (w\Id -S_0^{11})\, \xi \end{pmatrix}
\end{equation*}
\end{prop}

\begin{proof} By \eqref{eq:matchPPhiPsi1}, \eqref{eq:matchPPhiPsi2} and the 
last observation before the statement of the proposition, we find that 
for $\xi\in \C^{L+1}$ we have $\si, \tau\in \C^{L+1}$ so that 
\begin{equation*}
p\mapsto 
\Om \begin{pmatrix} \Phi_p(w^2)\, \si\\ \Psi_p(w^2)\, \tau \end{pmatrix} =
\begin{pmatrix} P_{2p}(w)\, \xi \\ P_{2p+1}(w)\, \xi \end{pmatrix}
\end{equation*}
in $\cF(\N; \C^{L+1}\oplus \C^{L+1})$. Taking $p=0$ and using 
$P_1(w)= (S^{12}_0)^{-1} (w\Id -S_0^{11})$  and $P_0(w)=\Phi_0(w^2)= \Psi_0(w^2)=\Id$ and $\Om$ as in Lemma \ref{lem:rangekernelT} shows that 
\begin{equation*}
\begin{pmatrix} \si\\ \tau \end{pmatrix} = \Om
\begin{pmatrix}  \xi \\ (S^{12}_0)^{-1} (w\Id -S_0^{11})\, \xi \end{pmatrix}
\end{equation*}
Plugging this in, gives the result. 
\end{proof}

Observe that the positive operators $JJ^\ast$ and $J^\ast J$ are each others Darboux transform and so they have the same spectrum (up to a possible exception for $0$), see Deift \cite[Thm.~2]{Deif}. We relate the polynomials $\Phi_n$ and $\Psi_n$ in Section \ref{sec:relationbetweenMOP} and Section \ref{sec:polardecompositions}. In Section \ref{sec:klr-constant-symmetric} we present an example where the nature of the spectrum indeed differs at $0$.

\begin{remark}\label{rmk:caseLis0:sec:MOPrelatedtoJ}
In case $L=0$, the polynomial families $(P_n)_{n\in \N}$,
$(\Phi_n)_{n\in \N}$, $(\Psi_n)_{n\in \N}$ are scalar 
orthogonal polynomials, and $JJ^\ast$ and $J^\ast J$ are 
Jacobi matrices. The polynomials $P_n$ are  
symmetric polynomials, i.e. $P_n(-x)=(-1)^nP_n(x)$, and Proposition \ref{prop:relationPnPhinPsin} describes the standard result that the even, respectively odd, degree polynomials can be expressed in terms of a related family of orthogonal polynomials, see e.g. \cite[\S I.8]{Chih}. 
Then $JJ^\ast$ and $J^\ast J$ are $UL$ and $LU$-factorisations of the Jacobi matrices, and the orthogonality measures are related via Christoffel and Geronimus transformations (at zero), see \cite[Ch.~6, 7]{GarcMM}, \cite[\S 2.7]{Isma}.
\end{remark}

\begin{remark}\label{rmk:caseLis1:sec:MOPrelatedtoJ}
In case $L=1$, putting $b_{n,1}=a_{n-1}$ brings the case 
studied by Pushnitski and \v{S}tampach \cite{PushS}, see also \cite{EichLY}. Then the three-term recurrence relation 
\eqref{eq:MOPPn} for the polynomials $P_n$ can be written 
uniformly, and the corresponding three-term recurrence relation matches the block Jacobi matrix in \cite[\S 2.8]{PushS}, \cite[\S 3]{EichLY}. The condition $b_{n,1}=a_{n-1}$ and the expressions for $A_p'$, $A''_p$, $B'_p$ and $B_p''$ 
in 
\eqref{eq:JJastonF} and \eqref{eq:JastJonF} show that 
$J$ is normal if and only if $\overline{b_{2p+2,0}} + b_{2p+1,0}$ is real for all $p$, or that the imaginary part of $b_{p,0}$ is independent of $p$. This is also observed in 
\cite[\S 2.2]{PushS}. Note that without the condition  
$b_{n,1}=a_{n-1}$ there is no easy condition ensuring that $J$ is a normal operator. 
\end{remark}

%%%%%%%%%%%%%%%%%%%%%%%%%%%%%%%%%
%%%%%%%%%%%%%%%%%%%%%%%%%%%%%%%%%
%%%%%%%%%%%%%%%%%%%%%%%%%%%%%%%%%
%%%%%%%%%%%%%%%%%%%%%%%%%%%%%%%%%

\section{Relation between matrix orthogonal polynomials}\label{sec:relationbetweenMOP}

In Proposition \ref{prop:relationPnPhinPsin} we have established a first relation between the families of matrix orthogonal polynomials. In this section we extend the relation between 
eigenfunctions of \eqref{eq:JJastonF}, \eqref{eq:JastJonF} 
and we relate the asymptotically free solutions for the 
polynomial recurrences. 

Observe that the positive operators $JJ^\ast$ and $J^\ast J$ are each others Darboux transform, so that $J$ maps eigenfunctions 
of $J^\ast J$ to those of $JJ^\ast$, see \cite{Deif}. 
Recall that $A'_p$, $B'_p$, $A''_p$, $B''_p$ are 
defined in \eqref{eq:JJastonF} and \eqref{eq:JastJonF}. 

\begin{lemma}\label{lem:JonFeigenfunctionJastJandback}
(i) Assume $F\in \cF(\N;\C^{L+1})$ satisfies 
\begin{equation*}
z \, F (p) = A'_p F(p+1) + B'_pF(p) + (A'_{p-1})^\ast F(p-1), \qquad 
p\geq k \geq 0, 
\end{equation*}
for some $k\in \N$ and all $z\in \C$, then $J^\ast F \in \cF(\N;\C^{L+1})$ satisfies  
\begin{equation*}
z\,  G (p) = A''_p G(p+1) + B''_pG(p) + (A''_{p-1})^\ast G(p-1), 
\end{equation*}
for all $p\geq k+1$.  If $k=0$, then it holds for all $p\geq 0$.
\par \noindent 
(ii) Assume $G\in \cF(\N;\C^{L+1})$ satisfies 
\begin{equation*}
z\,  G (p) = A''_p G(p+1) + B''_pG(p) + (A''_{p-1})^\ast G(p-1), \qquad p\geq k\geq 0
\end{equation*}
for some $k\in \N$ and all $z\in \C$, then 
$J G \in \cF(\N;\C^{L+1})$ satisfies  
\begin{equation*}
z \, F (p) = A'_p F(p+1) + B'_pF(p) + (A'_{p-1})^\ast F(p-1),
\end{equation*}
for all $p\geq k+1$ if $k\geq 0$ and for all $p\in \N$ if $k=0$. 
\end{lemma}

\begin{proof} We consider (i), and for this we consider in general 
\begin{equation}\label{eq:lem:JonFeigenfunctionJastJandback1}
\begin{split}
&\, A''_p \bigl(J^\ast F\bigr)(p+1) + B''_p \bigl(J^\ast F\bigr)(p) 
+ (A''_{p-1})^\ast  \bigl(J^\ast F\bigr)(p+1) 
\\ &\, = J^\ast \bigl( p\mapsto A'_p F(p+1) + B'_p F(p) + (A'_{p-1})^\ast F(p-1)\bigr)
\end{split}
\end{equation}
which follows from  $J^\ast J$ and $JJ^\ast$ being Darboux transforms. 
It is also straightforward to check \eqref{eq:lem:JonFeigenfunctionJastJandback1} 
directly using Corollary \ref{cor:lem:fromTonFNtoTonFNCN2},  
\eqref{eq:JastJonF} and \eqref{eq:JJastonF} and using $X_pZ^\ast_k=0=Z^\ast_k X_p$ for all $k$ and $p$.  Now  the assumption on $F$ gives that 
\begin{equation*}
J^\ast \bigl( p\mapsto A'_p F(p+1) + B'_p F(p) + (A'_{p-1})^\ast F(p-1)\bigr) 
= z \, (J^\ast F)(p), \qquad p\geq k+1
\end{equation*}
since we need to apply the assumption for $p$, $p+1$ and $p-1$. 
In case $k=0$, we have $X_{-1}=A'_{-1}=0$ and the result holds for all 
$p\geq 0$.

The proof of (ii) is analogous.
\end{proof}

The case $k=0$ of Lemma \ref{lem:JonFeigenfunctionJastJandback} gives a relation 
between the polynomials $\Phi_n$ and $\Psi_n$.

\begin{cor}\label{cor:lem:JonFeigenfunctionJastJandback} We have 
\begin{equation*}
\begin{split}
\bigl(J^\ast \Phi_\bullet(z)\bigr)(n) &\, = \Psi_n(z)\bigl( Z_1^\ast \Phi_1(z)+ Y_0^\ast\bigr), \\
\bigl(J \Psi_\bullet(z)\bigr)(n) &\, = \Phi_n(z) \bigl( X_0 \Psi_1(z)+ Y_0\bigr),
\end{split}
\end{equation*}
where the action of $J$ and $J^\ast$ is extended to $(L+1)\times (L+1)$-matrices. 
Moreover, 
\begin{equation*}
\begin{split}
z\, \Phi_n(z) &\, = \bigl( X_n \Psi_{n+1}(z) + Y_n \Psi_n(z) + Z_n \Psi_n(z)\bigr)
\bigl( Z_1^\ast \Phi_1(z)+ Y_0^\ast\bigr), \\
z\, \Psi_n(z) &\, = \bigl( Z_{n+1}^\ast \Phi_{n+1}(z) + Y_n^\ast \Phi_n(z) + X_{n-1}^\ast \Phi_n(z)\bigr)\bigl( X_0 \Psi_1(z)+ Y_0\bigr), 
\end{split}
\end{equation*}
so that 
\begin{align*}
\bigl( Z_1^\ast \Phi_1(z)+ Y_0^\ast\bigr) \bigl( X_0 \Psi_{1}(z) + Y_0 \bigr)=  z\, \Id &\, = \bigl( X_0 \Psi_{1}(z) + Y_0 \bigr)
\bigl( Z_1^\ast \Phi_1(z)+ Y_0^\ast\bigr).
\end{align*}
\end{cor}

\begin{proof} Define $F\in \cF(\N;\C^{L+1})$ by $F(p)= \Phi_p(z) \xi$ for 
$\xi \in \C^{L+1}$, then Lemma \ref{lem:JonFeigenfunctionJastJandback}(i) holds for $k=0$. Since all solutions are given by $\Psi_p(z)\eta$ for some $\eta\in \C^{L+1}$, we 
get $\bigl( J^\ast F\bigr)(p) = \Psi_p(z)\eta$. Taking $p=0$ and $\Phi_0(z)=\Id=\Psi_0(z)$  gives 
$\eta = \bigl( J^\ast F\bigr)(0) = Z_1^\ast \Phi_1(z)\xi + Y_0^\ast \xi$. 
Since $\xi\in \C^{L+1}$ is arbitrary, the result follows. The other case follows similarly. 

Now apply $J$ to the first equation and $J^\ast$ to the second, so that 
\eqref{eq:defPhiPsiMOP} gives the second result. Specalising $n=0$ gives the final statement. 
\end{proof}

Since $\Phi_1$, $\Psi_1$ are matrix polynomials of degree $1$, Corollary \ref{cor:lem:JonFeigenfunctionJastJandback} shows that 
\begin{equation}
\det \bigl( X_0 \Psi_{1}(z) + Y_0 \bigr) = C_1z^{a_1},  \ \ 
\det \bigl( Z_1^\ast \Phi_1(z)+ Y_0^\ast\bigr) = C_2 z^{a_2}, 
\quad a_1+a_2=L+1, \ C_1C_2=1
\end{equation}
for $a_1, a_2\in \N_{\geq 1}$. 

Corollary \ref{cor:lem:JonFeigenfunctionJastJandback} implies that 
\begin{equation}
\J \begin{pmatrix} \Phi_\bullet (w^2) T\eta \\ w \Psi_\bullet(w^2)\eta \end{pmatrix} 
= w \begin{pmatrix} \Phi_\bullet (w^2) T\eta \\ w \Psi_\bullet(w^2)\eta \end{pmatrix}, 
\qquad T=X_0\Psi_1(w^2) + Y_0, \ w\not=0,
\end{equation}
so that 
$\Om \begin{pmatrix} \Phi_p (w^2) T\eta \\ w \Psi_p(w^2)\eta \end{pmatrix}
= \begin{pmatrix} P_{2p}(w)\, \xi \\ P_{2p+1}(w)\, \xi \end{pmatrix}
$ and taking $p=0$ gives 
$\begin{pmatrix} T\eta \\ w \eta \end{pmatrix}
= \Om \begin{pmatrix} \xi \\ P_{1}(w)\, \xi \end{pmatrix}
$
in accordance with Proposition \ref{prop:relationPnPhinPsin}. 
Note that now the number of degrees of freedom is $L+1$ by 
Lemma \ref{lem:cyclicvectorsOmegabbJOmega} and Section \ref{ssec:differenceeqtnmatrixOP}, but there are $2(L+1)$ conditions. 
So we obtain a compatibility condition, which proves the first statement of Lemma \ref{lem:compatibilityPhiPsi}

\begin{lemma}\label{lem:compatibilityPhiPsi} 
For $w\in \C$ and $\xi = \sum_{i=0}^L \xi_i v_i\in \C^{L+1}$ we have 
\begin{equation*} 
(Z_1^\ast \Phi_1(w^2) + Y_0^\ast) 
\begin{pmatrix}
\xi_0\\ \xi_1\\ \vdots \\ \xi_{L-1} \\ (P_1(w)\xi)_0 
\end{pmatrix}
= 
w
\begin{pmatrix}
\xi_L\\ (P_1(w)\xi)_1\\ \vdots \\ (P_1(w)\xi)_{L-1} \\ (P_1(w)\xi)_L 
\end{pmatrix} 
\end{equation*}
and 
\begin{equation*} 
(X_0\Psi_1(w^2) + Y_0)
\begin{pmatrix}
\xi_L\\ (P_1(w)\xi)_1\\ \vdots \\ (P_1(w)\xi)_{L-1} \\ (P_1(w)\xi)_L 
\end{pmatrix} 
= w
\begin{pmatrix}
\xi_0\\ \xi_1\\ \vdots \\ \xi_{L-1} \\ (P_1(w)\xi)_0 
\end{pmatrix}. 
\end{equation*}
\end{lemma}

\begin{proof} The vector in $\C^{L+1}$ on the right hand side of the first equation 
equals the last $L+1$-components of 
$\begin{pmatrix} T\eta \\ w \eta \end{pmatrix}
= \Om \begin{pmatrix} \xi \\ P_{1}(w)\, \xi \end{pmatrix} 
\in \C^{2(L+1)}
$
and so equals $w\eta$. The vector in $\C^{L+1}$ on the left hand side of the first equation 
equals the last $L+1$-components of 
$\begin{pmatrix} T\eta \\ w \eta \end{pmatrix}
= \Om \begin{pmatrix} \xi \\ P_{1}(w)\, \xi \end{pmatrix} 
\in \C^{2(L+1)}
$ and equals $T\eta$. Rewriting this using the last statement of Corollary \ref{cor:lem:JonFeigenfunctionJastJandback} gives the first statement. Initially, this derivation holds for $w\not=0$, but since it is a polynomial identity in $w$ it follows for $w=0$ as well. 

The second statement follows from the first using the 
last statement of Corollary \ref{cor:lem:JonFeigenfunctionJastJandback}. 
\end{proof}

\begin{remark}
Note that we have used Lemma \ref{lem:JonFeigenfunctionJastJandback} only for the simplest case $k=0$.
More generally, Lemma \ref{lem:JonFeigenfunctionJastJandback} can also be used to 
calculate the action of $J^\ast$ and $J$ on the asymptotically 
free solutions to \eqref{eq:JJastonF} and on \eqref{eq:JastJonF}. 
This gives a way of relating the associated polynomials 
$\Phi_n^{(1)}$ to $\Psi_n^{(1)}$, but one also needs the 
polynomials $\Psi_n$ as well as the second associated polynomials. We will not need this here, since we relate the Stieltjes transforms of the measures in another way to each other, see Corollary  \ref{cor:thm:weightsfrompartialisometryforJandJast}. 
\end{remark}

By $P_n^{(1)}$ we denote the associated matrix polynomials for the matrix polynomials $P_n$, i.e. $P_n^{(1)}$ solves \eqref{eq:MOPPn} with initial conditions
$P_0^{(1)}(z)=0$ and $P_1^{(1)}(z)=A_0^{-1}=(S_0^{12})^{-1}$. Similarly, the associated matrix polynomials 
for $\Phi_n$, respectively $\Psi_n$, are the matrix polynomials 
$\Phi_n^{(1)}$, respectively $\Psi_n^{(1)}$, satsifying the first, respectively
second, equation of \eqref{eq:defPhiPsiMOP} with initial conditions
$\Phi_0^{(1)}(z)=0$ and $\Phi_1^{(1)}(z)=(A_0')^{-1}$, respectively
$\Psi_0^{(1)}(z)=0$ and $\Psi_1^{(1)}(z)=(A_0'')^{-1}$. Note that the associated 
polynomials do not satisfy \eqref{eq:MOPPn} and \eqref{eq:defPhiPsiMOP} for $n=0$,
but that they satisfy, see Section \ref{ssec:differenceeqtnmatrixOP}, 
\begin{equation}\label{eq:3termrecurrenceforassociatedpols}
\begin{split}
\de_{n,0}\Id+ z\,P^{(1)}_n(z) &\, = A_n \, P^{(1)}_{n+1}(z) + B_n \, P^{(1)}_n(z) + A_{n-1}^\ast P^{(1)}_{n-1}(z), \\
\de_{n,0}\Id+ z \, \Phi^{(1)}_n(z) &\, = A'_n\, \Phi^{(1)}_{n+1}(z) + B'_n\, \Phi^{(1)}_n(z) + (A'_{n-1})^\ast \Phi^{(1)}_{n-1}(z), \\
\de_{n,0}\Id+ z \, \Psi^{(1)}_n(z) &\, = A''_n\, \Psi^{(1)}_{n+1}(z) + B''_n\, \Psi^{(1)}_n(z) + (A''_{n-1})^\ast \Psi^{(1)}_{n-1}(z).
\end{split}
\end{equation}

Now observe that 
\[
\Bigl( \Om\J \Om 
\begin{pmatrix}
P_{2\bullet}^{(1)}(w)\, \xi \\[1mm] P_{2\bullet +1}^{(1)}(w)\, \xi 
\end{pmatrix} \Bigr)(p)
= w
\begin{pmatrix}
P_{2p}^{(1)}(w)\, \xi \\[1mm] P_{2p+1}^{(1)}(w)\, \xi 
\end{pmatrix}
+ \de_{p,0} 
\begin{pmatrix} \xi \\[1mm] 0 \end{pmatrix}
\]
so that 
\begin{equation}\label{eq:OmJ2Omassociatedpols}
\begin{split}
\Bigl( \Om\J^2 \Om 
\begin{pmatrix}
P_{2\bullet}^{(1)}(w)\, \xi \\[1mm] P_{2\bullet +1}^{(1)}(w)\, \xi 
\end{pmatrix} \Bigr)(p)
&\, = w \Bigl(
\begin{pmatrix}
P_{2p}^{(1)}(w)\, \xi \\[1mm] P_{2p+1}^{(1)}(w)\, \xi 
\end{pmatrix}
+ \de_{p,0} 
\begin{pmatrix} \xi \\[1mm] 0 \end{pmatrix}\Bigr) 
+ \bigl( \Om \J \Om (
\de_{\bullet,0} 
\begin{pmatrix} \xi \\[1mm] 0 \end{pmatrix} )
\bigr)(p) \\
&\, = w^2 
\begin{pmatrix}
P_{2p}^{(1)}(w)\, \xi \\[1mm] P_{2p+1}^{(1)}(w)\, \xi 
\end{pmatrix}
+ \de_{p,0} 
\begin{pmatrix} (w+S_0^{11})\xi \\[1mm] S^{21}_0\, \xi \end{pmatrix}.
\end{split}
\end{equation}

\begin{prop}\label{prop:relationassociatedPnPhinPsin} We have 
\begin{align*}
\Om 
\begin{pmatrix} P_{2p}^{(1)}(w)\, \xi \\[1mm]
                           P_{2p+1}^{(1)}(w)\, \xi 
                           \end{pmatrix}
= &\  
\begin{pmatrix} \Phi_p(w^2) & 0 \\[1mm] 0 & \Psi_p(w^2) \end{pmatrix} \Om 
\begin{pmatrix}  0 \\[1mm] (S^{12}_0)^{-1} \, \xi \end{pmatrix} 
\\ &\, +  
\begin{pmatrix} \Phi_p^{(1)}(w^2) & 0 \\[1mm] 0 & \Psi_p^{(1)}(w^2) \end{pmatrix} \Om 
\begin{pmatrix} (w\Id +S_0^{11})\xi \\[1mm] S^{21}_0\, \xi \end{pmatrix}
\end{align*}
\end{prop}

\begin{proof} Using \eqref{eq:defPhiPsiMOP} and \eqref{eq:3termrecurrenceforassociatedpols}, we note that for 
$\si, \tau, \si^{(1)}, \tau^{(1)}\in \C^{L+1}$ 
\[
\Bigl( 
\J^2 \begin{pmatrix} \Phi_\bullet(z)\si + \Phi_\bullet^{(1)}(z)\si^{(1)} \\ \Psi_\bullet(z)\tau + \Psi_\bullet^{(1)}(z)\tau^{(1)} \end{pmatrix}  \Bigr)(p)
= z \begin{pmatrix} \Phi_p(z)\si + \Phi_p^{(1)}(z)\si^{(1)} \\ \Psi_p(z)\tau + \Psi_p^{(1)}(z)\tau^{(1)} \end{pmatrix} 
+ \de_{p,0} \begin{pmatrix} \si^{(1)} \\ \tau^{(1)} \end{pmatrix}
\]
gives the general solution, so that comparing with 
\eqref{eq:OmJ2Omassociatedpols} gives 
\begin{equation*}
\begin{pmatrix}
P_{2p}^{(1)}(w)\, \xi \\[1mm] P_{2p+1}^{(1)}(w)\, \xi 
\end{pmatrix} = \Om 
\begin{pmatrix} \Phi_p(w^2)\si + \Phi_p^{(1)}(w^2)\si^{(1)} \\ \Psi_p(w^2)\tau + \Psi_p^{(1)}(w^2)\tau^{(1)} \end{pmatrix}, \qquad 
\begin{pmatrix} (w\Id +S_0^{11})\xi \\[1mm] S^{21}_0\, \xi \end{pmatrix}
= 
\Om \begin{pmatrix} \si^{(1)} \\ \tau^{(1)} \end{pmatrix},
\end{equation*}
and it suffices to evaluate $\si, \si^{(1)}, \tau,\tau^{(1)} \in \C^{L+1}$ from 
this condition for $p=0$. Since $P_1^{(1)}(z) = (S_0^{12})^{-1}$, this gives
\begin{equation*}
\Om \begin{pmatrix}
0 \\[1mm] (S_0^{12})^{-1}\, \xi 
\end{pmatrix} = 
\begin{pmatrix} \si  \\ \tau \end{pmatrix}, \qquad 
\Om \begin{pmatrix} (w\Id +S_0^{11})\xi \\[1mm] S^{21}_0\, \xi \end{pmatrix}
= \begin{pmatrix} \si^{(1)} \\ \tau^{(1)} \end{pmatrix},
\end{equation*}
and this gives the result.
\end{proof}

Next we want to connect the Stieltjes transforms of the measures. Recall that we assume bounded coefficients, i.e. 
$\sup_{n\in \N} \bigl( |a_n| + \sum_{k=0}^L |b_{n,k}|\bigr) <\infty$. Then all the operators involved are bounded operators, and in particular $JJ^\ast$ and $J^\ast J$ as 
\eqref{eq:JJastonF} and \eqref{eq:JastJonF} are bounded, positive,  
self-adjoint operators. 

The $(L+1)\times(L+1)$-matrix polynomial families  $(P_n)_{n\in\N}$,  
$(\Phi_n)_{n\in\N}$, and $(\Psi_n)_{n\in\N}$ are orthonormal matrix polynomials 
with respect to the matrix measures $\rho$, $\mu$ and $\nu$. Explicitly, 
\begin{equation}\label{eq:defMmeasuresrhomunuPnPhinPsin}
\begin{split}
\int_\R P_n(x)\,  d\rho(x)\, (P_n(x))^\ast  &\, = 
\int_\R P_n(x) W^\rho(x) (P_n(x))^\ast \, d\rho_\Tr(x)=\de_{m,n} \Id, \\ 
\int_\R \Phi_n(x)\,  d\mu(x)\, (\Phi_n(x))^\ast  &\, = 
\int_\R \Phi_n(x) W^\mu(x) (\Phi_n(x))^\ast \, d\mu_\Tr (x) =\de_{m,n} \Id, \\ 
\int_\R \Psi_n(x)\,  d\nu(x)\, (\Psi_n(x))^\ast  &\, = 
\int_\R \Psi_n(x) W^\nu(x) (\Psi_n(x))^\ast \, d\nu_\Tr (x) =\de_{m,n} \Id. 
\end{split}
\end{equation}
The notation follows from the discussion in Section \ref{sec:recapMOP}. 
Let $M^\rho$, $M^\mu$, $M^\nu$ be the respective Stieltjes transforms of the 
matrix measures involved as in Section \ref{sec:recapMOP}. Note that 
$\supp(\mu)=\supp(\mu_\Tr)\subset [0,\infty)$ and 
$\supp(\nu)=\supp(\nu_\Tr)\subset [0,\infty)$, since this corresponds to the spectrum of the positive operators $JJ^\ast$ and $J^\ast J$. Moreover, 
as $JJ^\ast$ and $J^\ast J$ are Darboux transformations, we have  
$\supp(\mu_\Tr) \setminus\{0\} =\supp(\nu_\Tr) \setminus\{0\}$, see Deift \cite[Thm.~2]{Deif}. 
In the next result we relate the Stieltjes transforms of these matrix measures.

\begin{thm}\label{thm:relationStieltjestransforms} 
The Stieltjes transforms $M^\rho$, $M^\mu$ and $M^\nu$ are related via 
\begin{align*}
\begin{pmatrix}
M^\mu(w^2) & 0 \\
0 & M^\nu(w^2)
\end{pmatrix}
\Om 
\begin{pmatrix}
w\Id + S^{11}_0  & 0 \\
0 & S_0^{21}
\end{pmatrix} 
= \Om 
\begin{pmatrix}
M^\rho(w) & 0 \\
0 & (S_0^{12})^{-1}+P_1(w) M^\rho (w)
\end{pmatrix}
\end{align*}
for all $w\in \C\setminus\R$. 
\end{thm}

Note that the Stieltjes transforms $M^\mu$ and $M^\nu$ are well defined on $(-\infty, 0)$, since the support of the measures $\mu$ and $\nu$ is contained in $[0,\infty)$. 
In Corollary \ref{cor:thm:weightsfrompartialisometryforJandJast} we connect the Stieltjes transforms for the matrix measures $\nu$ and $\mu$.

\begin{proof} Since the operators involved are bounded, we know from Section \ref{ssec:differenceeqtnmatrixOP} that the asymptotically free solutions for the recurrence relation corresponding to the matrix polynomials $P_n$ satisfy
$\sum_{p=0}^\infty \| \bigl( P^{(1)}_p(w) + P_p(w) M^\rho(w)\bigr) \xi\|^2_{\C^{L+1}}<\infty$. This then gives 
\begin{align*}
\infty > &\, \sum_{p=0}^\infty \| \begin{pmatrix} \bigl( P^{(1)}_{2p}(w) + P_{2p}(w) M^\rho(w)\bigr) \xi \\ 
\bigl( P^{(1)}_{2p+1}(w) + P_{2p+1}(w) M^\rho(w)\bigr) \xi                     
\end{pmatrix}
\|^2_{\C^{2(L+1)}} \\
= &\, \sum_{p=0}^\infty \| \begin{pmatrix} 
\Phi_p(w^2) & 0 \\ 0 & \Psi_p(w^2)
\end{pmatrix} \Om 
\begin{pmatrix}
M^\rho(w)\xi \\ 
((S_0^{12})^{-1} + P_1(w)M^\rho(w))\xi
\end{pmatrix}
+ 
\\ &\, \qquad\qquad\qquad\qquad\qquad\qquad\qquad\qquad
\begin{pmatrix} 
\Phi^{(1)}_p(w^2) & 0 \\ 0 & \Psi^{(1)}_p(w^2)
\end{pmatrix} \Om 
\begin{pmatrix}
(w\Id +S_0^{11})\xi \\ 
S_0^{21} \xi
\end{pmatrix}
\|^2_{\C^{2(L+1)}},
\end{align*}
where we use Proposition \ref{prop:relationPnPhinPsin} and 
Proposition \ref{prop:relationassociatedPnPhinPsin}, 
$P_1(w) = (S_0^{12})^{-1}(w\Id -S_0^{11})$ and $\Om$ being an isometry.
For the first $L+1$ components we get 
\begin{equation}\label{eq:thmStieltjesPhi}
\sum_{p=0}^\infty 
\| \Phi_p(w^2) R_\Phi(w) \xi +  \Phi^{(1)}_p(w^2) R^{(1)}_\Phi(w) \xi \|_{\C^{L+1}}^2 <\infty,
\end{equation}
where 
\begin{equation*}
R_\Phi(w) \xi =
\begin{pmatrix}
(M^\rho(w)\xi)_0 \\ \vdots \\ (M^\rho(w)\xi)_{L-1} \\
\bigl( ((S_0^{12})^{-1} + P_1(w) M^\rho(w))\xi\bigr)_0 
\end{pmatrix}, \quad 
R^{(1)}_\Phi(w) \xi =
\begin{pmatrix}
((w\Id  + S_0^{11})\xi)_0 \\ \vdots \\ ((w\Id  + S_0^{11})\xi)_{L-1} \\
( S_0^{21}\xi)_0 
\end{pmatrix}.
\end{equation*}
Here the maps $R_\Phi(w)$ and $R^{(1)}_\Phi(w)$ do not depend on $p$, but they depend on $w\in \C\setminus\R$.  
So \eqref{eq:thmStieltjesPhi} gives the asymptotically free solutions for the recurrence relation related to the polynomials $\Phi_n$. 
Assuming $R^{(1)}_\Phi(w)$ invertible for the moment for $w\not=0$, we see that 
\begin{equation}\label{eq:thmStieltjesPhi1}
\sum_{p=0}^\infty 
\| \Phi_p(w^2) R_\Phi(w) (R^{(1)}_\Phi(w))^{-1} \eta +  \Phi^{(1)}_p(w^2) \eta \|_{\C^{L+1}}^2 <\infty, 
\end{equation}
so that $M^\mu(w^2) = R_\Phi(w) (R^{(1)}_\Phi(w))^{-1}$ is the Stieltjes transform for the matrix measure $\mu$ in $w^2$ for $w\not=0$ by 
comparing \eqref{eq:thmStieltjesPhi1} with Section 
\ref{ssec:differenceeqtnmatrixOP}.  
Using the explicit 
expressions for $S_0^{11}$ and $S_0^{21}= (S_0^{12})^\ast$ from Lemma
\ref{lem:rangekernelT} and  \eqref{eq:defS11andS22}, we see that 
\[
R^{(1)}_\Phi(w) = 
\begin{pmatrix}
w & 0 & \cdots & 0 & b_{0,0} \\
0 & w & \cdots & 0 & b_{1,1} \\
\vdots &  & \ddots &  & \vdots \\
0 &  & \ddots & w & b_{L-1,L-1} \\
0 &  & \cdots & 0 & b_{L,L}
\end{pmatrix}, 
\qquad  \det \bigl( R^{(1)}_\Phi(w) \bigr) = w^L b_{L,L}
\]
and invertibility for $w\not=0$ follows. 

Similarly, considering the last $\C^{L+1}$ components, the analogue of \eqref{eq:thmStieltjesPhi} for $\Psi_n$ is
\begin{equation}\label{eq:thmStieltjesPsi}
\sum_{p=0}^\infty 
\| \Psi_p(w^2) R_\Psi(w) \xi +  \Psi^{(1)}_p(w^2) R^{(1)}_\Psi(w) \xi \|_{\C^{L+1}}^2 <\infty, 
\end{equation}
where
\begin{equation*} 
R_\Psi(w) \xi =
\begin{pmatrix}
(M^\rho(w)\xi)_L \\ \bigl( ((S_0^{12})^{-1} + P_1(w)M^\rho(w))\xi\bigr)_1 \\ \vdots \\
\bigl( ((S_0^{12})^{-1} + P_1(w)M^\rho(w))\xi\bigr)_L
\end{pmatrix}, \quad 
R^{(1)}_\Psi(w) \xi =
\begin{pmatrix}
((w\Id  + S_0^{11})\xi)_L \\ ( S_0^{21}\xi)_0 \\ \vdots  \\
( S_0^{21}\xi)_L 
\end{pmatrix}.
\end{equation*}
Now we have a slightly more complicated matrix 
\[
R^{(1)}_\Psi(w) = 
\begin{pmatrix}
\overline{b_{0,0}} & \overline{b_{1,1}} & \cdots & \cdots &  \overline{b_{L-1,L-1}} & w \\
a_0                & \overline{b_{1,0}} & \cdots & \cdots & \overline{b_{L-1,L-2}} & 0 \\
 0                 & a_1                & \overline{b_{2,0}} & \cdots & \overline{b_{L-1,L-3}} & 0 \\
\vdots &  & \ddots & \ddots  & \vdots & \vdots \\
0 & \cdots &  & a_{L-2} & \overline{b_{L-1,0}} & 0 \\
0 & \cdots &  & 0 & a_{L-1} & 0
\end{pmatrix}, 
\qquad  \det \bigl( R^{(1)}_\Psi(w) \bigr) = w \prod_{p=0}^{L-1} a_p, 
\]
so that for $w\not=0$ we have 
$M^\nu(w^2) = R_\Psi(w) (R^{(1)}_\Psi(w))^{-1}$ for the Stieltjes transform of the matrix measure $\nu$. 

Combining the results for the Stieltjes transform of the matrix measures $\mu$ and $\nu$ gives 
\begin{align*}
&\,
\begin{pmatrix}
M^\mu(w^2) & 0 \\
0 & M^\nu(w^2)
\end{pmatrix}
\begin{pmatrix}
R^{(1)}_\Phi(w)& 0 \\
0 & R^{(1)}_\Psi(w) 
\end{pmatrix} 
= \begin{pmatrix}
R_\Phi(w)& 0 \\
0 & R_\Psi(w) 
\end{pmatrix} \quad \Longrightarrow  \\
&\,
\begin{pmatrix}
M^\mu(w^2) & 0 \\
0 & M^\nu(w^2)
\end{pmatrix}
\Om 
\begin{pmatrix}
w\Id + S^{11}_0  & 0 \\
0 & S_0^{21}
\end{pmatrix} \Om 
= \Om 
\begin{pmatrix}
M^\rho(w) & 0 \\
0 & (S_0^{12})^{-1}+P_1(w) M^\rho (w)
\end{pmatrix}
\Om
\end{align*}
and this gives the result. 
\end{proof}

\begin{remark}\label{rmk:caseLis0:sec:relationbetweenMOP}
Taking $L=0$ as in Remark 
\ref {rmk:caseLis0:sec:MOPrelatedtoJ}, 
obtain the relation $M^\mu(z)|b_0|^2 = 1+zM^\nu(z)$, which 
can be derived directly using the results on the 
Christoffel and Geronimus transformations (at zero), see \cite[Ch.~6, 7]{GarcMM}, \cite[\S 2.7]{Isma}.
\end{remark}

\begin{remark}\label{rmk:caseLis1:sec:relationbetweenMOP}
In case $L=1$, putting $b_{n,1}=a_{n-1}$, 
see Remark \ref{rmk:caseLis1:sec:MOPrelatedtoJ}, the 
Stieltjes transforms are studied in \cite{EichLY} in more detail, and also for the unbounded setting.  
Note that Eichinger et al. use the terminology Weyl matrix for the Stieltjes transform. In particular, the Stieltjes transform $M^\rho$ (or of $\J$) is symmetric, see \cite[Def.~3.3]{EichLY}, and it is completely determined 
by a measure, which is essentially the $(0,0)$-entry of the 
matrix measure, and a bounded function, see \cite{PushS} and \cite{EichLY} for details. 
\end{remark}

%%%%%%%%%%%%%%%%%%%%%%%%%%%%%%%%%
%%%%%%%%%%%%%%%%%%%%%%%%%%%%%%%%%
%%%%%%%%%%%%%%%%%%%%%%%%%%%%%%%%%
%%%%%%%%%%%%%%%%%%%%%%%%%%%%%%%%%

\section{Polar decompositions}\label{sec:polardecompositions}

The purpose of this section is to give an explicit expression for the polar decompositions of $J$ and $J^\ast$ as bounded operators on $\ell^2(\N;\C^{L+1})$. We connect the polar decomposition to the Stieltjes transform of the measure. 
Recall, see e.g. \cite[Thm.~7.20, Exercise 7.26]{Weid}, that 
$J= U |J|$, where $|J|=\sqrt{J^\ast J}$, and 
$J^\ast = U^\ast |J^\ast|$, where $|J^\ast|=\sqrt{JJ^\ast}$. Here 
$U$ is the partial isometry with initial projection onto
$\overline{\Ran (|J|)}$ and 
final projection onto $\overline{\Ran J}$ defined by 
$U \colon |J|F \mapsto JF$ for $F\in \ell^2(\N;\C^{L+1})$. 
Since $|J|^2= J^\ast J$ and $|J^\ast|^2= JJ^\ast$, we have the 
spectral decomposition for $|J|$ and $|J^\ast|$ as multiplication operators by $\sqrt{x}$ on $L^2(\nu;\C^{L+1})$ and on $L^2(\mu;\C^{L+1})$. 
Recall that $\supp(\nu)$ and $\supp(\mu)$ are contained in $[0,\infty)$,
since $J^\ast J$ and $JJ^\ast$ are positive bounded operators, 
so that multiplication by $\sqrt{x}$ is well-defined. 

Explicitly, see Section \ref{ssec:differenceeqtnmatrixOP}, we have $\cF \colon \ell^2(\N) \to L^2(\nu;\C^{L+1})$, respectively 
$\cG \colon \ell^2(\N) \to L^2(\mu;\C^{L+1})$, as unitary maps so that 
$M \cF = J^\ast J \cF$ and $M \cG = JJ^\ast \cG$, where $M$ denotes the 
multiplication operator by $x$ on $L^2(\nu;\C^{L+1})$, respectively 
$L^2(\mu;\C^{L+1})$. Next the partial isometry 
$\cU\colon L^2(\nu;\C^{L+1})  \to  L^2(\nu;\C^{L+1})$ is defined using the diagram 
\begin{center}
\begin{tikzcd}
\ell^2(\N;\C^{L+1})  \arrow[d, "\cF"] \arrow[r, "U"]
&  \ell^2(\N;\C^{L+1}) \arrow[d, "\cG"] \\
L^2(\nu;\C^{L+1})   \arrow[r, "\cU"] 
& L^2(\mu;\C^{L+1}) 
\end{tikzcd}
\end{center}

\begin{prop}\label{prop:partialisometryforJandJast} The action of 
$\cU\colon L^2(\nu;\C^{L+1})  \to  L^2(\mu;\C^{L+1})$ is given by 
\begin{equation*}
L^2(\nu;\C^{L+1}) \ni (\la\mapsto 
\sqrt{\la} \Psi_k(\la)^\ast \xi) \mapsto 
(\la \mapsto (Y_0+\Phi_1(\la)^\ast Z_1) \Psi_k(\la)^\ast \xi )
\in L^2(\mu;\C^{L+1})
\end{equation*}
and the action of its adjoint $\cU^\ast \colon L^2(\mu;\C^{L+1})  \to  
L^2(\nu;\C^{L+1})$ is given by 
\begin{equation*}
L^2(\mu;\C^{L+1}) \ni (\la\mapsto 
\sqrt{\la} \Phi_k(\la)^\ast \xi) \mapsto 
(\la \mapsto (Y_0^\ast + \Psi_1(\la)^\ast X_0^\ast) \Phi_k(\la)^\ast \xi) 
\in L^2(\nu;\C^{L+1})
\end{equation*}
for $k\in \N$ and $\xi \in \C^{L+1}$. 
\end{prop}

\begin{proof} Since 
$\cG J \cF^\ast  = \cG U |J| \cF^\ast 
= \cG U  \cF^\ast M_{\sqrt{\cdot}} 
= \cU M_{\sqrt{\cdot}}$, and recalling, see Section \ref{ssec:differenceeqtnmatrixOP}, that 
$\cF^\ast \Psi_k(\cdot)^\ast \xi = F_{k,\xi}$, we see that 
we need $J F_{k,\xi}$. A straightforward calculation, see 
Section \ref{ssec:differenceeqtnmatrixOP},  gives 
\begin{equation*}
JF_{k,\xi} = F_{k-1, X_{k-1}\xi} + F_{k, Y_k\xi} + 
F_{k+1, Z_{k+1}\xi} 
\end{equation*}
and applying $\cG$ to the right hand side we get, again using Section 
\ref{ssec:differenceeqtnmatrixOP}, as an element of $L^2(\mu;\C^{L+1})$, 
\begin{align*}
\Phi_{k-1}(\cdot)^\ast X_{k-1}\xi + \Phi_k(\cdot)^\ast Y_k\xi + 
\Phi_{k+1}(\cdot)^\ast Z_{k+1}\xi  = 
\Bigl( X_{k-1}^\ast \Phi_{k-1}(\cdot) + Y_k^\ast \Phi_k(\cdot) + 
Z_{k+1}^\ast \Phi_{k+1}(\cdot)\Bigr)^\ast \xi 
\end{align*}
and the term in parentheses is $\bigl( J^\ast \Phi_\bullet(\cdot)\bigr)(k) 
= \Psi_k(\cdot)\bigl( Y_0^\ast + Z_1^\ast \Phi_1(\cdot)\bigr)$ 
by Corollary \ref{cor:lem:JonFeigenfunctionJastJandback}. 
Taking the adjoint, and letting the expression act on $\xi\in \C^{L+1}$ we get the first statement. 

Since the polar decomposition for $J^\ast = U^\ast |J^\ast|$, we can 
proceed as above switching from $J$ to $J^\ast$ to obtain the second statement. 
\end{proof}

Proposition \ref{prop:partialisometryforJandJast} gives the possibility to 
connect the orthogonality matrix measures $\mu$ and $\nu$ for the 
matrix polynomial families  $(\Phi_n)_{n\in \N}$ and $(\Psi_n)_{n\in \N}$. 

\begin{thm}\label{thm:weightsfrompartialisometryforJandJast}
Let $f, g$ be $\C^{L+1}$-valued polynomials,  then  
\begin{multline*}
\int_0^\infty g^\ast(\la)\,   W^{\nu}(\la) \, f(\la)\, \la\, 
d\nu_\Tr(\la) \\ =
\int_0^\infty g^\ast(\la)\, 
\bigl( Y_0^\ast + Z_1^\ast \Phi_1(\la)\bigr) W^{\mu}(\la) 
\bigl( Y_0 +  \Phi_1(\la)^\ast Z_1\bigr) \, f(\la)\, 
d\mu_\Tr(\la) 
\end{multline*}
and 
\begin{multline*}
\int_0^\infty g^\ast(\la)\,   W^{\mu}(\la) \, f(\la)\, \la\, 
d\mu_\Tr(\la) \\ =
\int_0^\infty g^\ast(\la)\, 
\bigl( Y_0 + X_0 \Psi_1(\la)\bigr) W^{\nu}(\la) 
\bigl( Y_0^\ast +  \Psi_1(\la)^\ast X_0^\ast \bigr) \, f(\la)\, 
d\nu_\Tr(\la).
\end{multline*}
\end{thm}

Note that we can take e.g. $f(\la) = \la f_1(\la)$ in the 
integrals on the left hand side, so that there is a factor $\la^2$. Then we can next apply the other integral equation as well, and we obtain the same integral because of the last equality of Corollary \ref{cor:lem:JonFeigenfunctionJastJandback}. So the two 
results of Theorem \ref{thm:weightsfrompartialisometryforJandJast} are compatible. 

\begin{proof} First assume that $\Ker |J| =\Ker J$ is trivial, so that 
$U^\ast U =\Id_{\ell^2(\N;\C^{L+1})}$, and consequently 
$\cU^\ast \cU = \Id_{L^2(\nu;\C^{L+1})}$. Then we get for $k,l\in \N$ and 
$\xi,\eta\in \C^{L+1}$ that 
\begin{align*}
&\,  \int_0^\infty \eta^\ast \Psi_l(\la)\,   W^{\nu}(\la) \, \Psi_k(\la)^\ast \xi \, \la\, 
d\nu_\Tr(\la) \\ = &\,  
\langle \sqrt{\cdot} \Psi_k(\cdot)^\ast \xi ,
\sqrt{\cdot} \Psi_l(\cdot)^\ast \eta \rangle_{L^2(\nu;\C^{L+1})}
\\  = &\, \langle \cU^\ast \cU \sqrt{\cdot} \Psi_k(\cdot)^\ast \xi ,
\sqrt{\cdot} \Psi_l(\cdot)^\ast \eta \rangle_{L^2(\nu;\C^{L+1})} \\ =&\, 
\langle \cU \sqrt{\cdot} \Psi_k(\cdot)^\ast \xi ,
\cU \sqrt{\cdot} \Psi_l(\cdot)^\ast \eta \rangle_{L^2(\mu;\C^{L+1})} \\
= &\, 
\langle (Y_0+\Phi_1(\cdot)^\ast Z_1) \Psi_k(\cdot)^\ast \xi,
(Y_0+\Phi_1(\cdot)^\ast Z_1) \Psi_l(\cdot)^\ast \eta \rangle_{L^2(\mu;\C^{L+1})} \\
=&\, 
\int_0^\infty \eta^\ast \Psi_l(\la) \,  
\bigl( Y_0^\ast + Z_1^\ast \Phi_1(\la)\bigr) W^{\mu}(\la) 
\bigl( Y_0 +  \Phi_1(\la)^\ast Z_1\bigr) \, \Psi_k(\la)^\ast \xi \, 
d\mu_\Tr(\la) \ ,
\end{align*}
using the first statement of Proposition \ref{prop:partialisometryforJandJast}. This proves the statement for 
$g=\Psi_l(\cdot)^\ast \eta$ and $f=\Psi_k(\cdot)\xi$. 
Since all vector polynomials are linear combinations of such polynomials, the result follows.  
In case $\Ker|J|= \Ker J$ is not trivial, we have that 
$\nu(\{0\})\in M_{L+1}(\C)$ is a non-trivial non-negative definite matrix. Then $\sqrt{\cdot} \Psi_k(\cdot)^\ast \xi \in \cF (\Ker |J|^\perp)$ because of the square root, and the above calculation remains valid. 

The second statement follows analogously using the second statement of Proposition \ref{prop:partialisometryforJandJast}. 
\end{proof}

Recall that the Stieltjes transform of a matrix measure, as defined in Section \ref{ssec:differenceeqtnmatrixOP}, is a 
generating function for the matrix moments of the corresponding matrix measure. We use the notation for the Stieltjes transform as in Theorem \ref{thm:relationStieltjestransforms}. 
In particular, Theorem \ref{thm:weightsfrompartialisometryforJandJast}, can be used to give a connection between the Stieltjes transform 
for the matrix measures $\mu$ and $\nu$.

\begin{cor}\label{cor:thm:weightsfrompartialisometryforJandJast}
The Stieltjes transforms $M^\mu$ and $M^\nu$ satisfy 
\begin{align*}
zM^\nu(z)+\Id = &\, 
\bigl( Y_0^\ast +Z_1^\ast \Phi_1(z)\bigr)  M^\mu(z) \bigl( 
Y_0 + \Phi_1^\ast (z)Z_1\bigr) + R(z), \\
zM^\mu(z)+\Id = &\, 
\bigl( Y_0 +X_0 \Psi_1(z)\bigr)  M^\nu(z) \bigl( 
Y_0^\ast  + \Psi_1^\ast (z)X_0^\ast \bigr) + S(z) , 
\end{align*}
where $R$ and $S$ are degree one polynomials explicitly given by
\begin{align*}
R(z) = &\, Y_0^\ast ((A_0')^{-1})^\ast Z_1 + 
Z_1^\ast (A_0')^{-1} Y_0 + Z_1^\ast \Phi_1(z) ((A_0')^{-1})^\ast Z_1, \\
S(z) = &\, Y_0((A_0'')^{-1})^\ast X_0^\ast + 
X_0(A_0'')^{-1} Y_0^\ast + 
X_0\Psi_1(z) ((A_0'')^{-1})^\ast X_0^\ast , 
\end{align*}
where $R^\ast(z) = R(z)$, $S^\ast(z) = S(z)$. 
\end{cor}

Note that multiplying the first identity by $z$ and using the second identity and the last result of Corollary \ref{cor:lem:JonFeigenfunctionJastJandback} shows that 
\begin{equation}
(X_0\Psi_1(z) + Y_0)(R(z)-z\Id) = (zS(z) - z\Id ) (Y_0+\Phi_1^\ast(z)Z_1) 
\end{equation}
is the required compatibility condition, where we use Corollary 
\ref{cor:lem:JonFeigenfunctionJastJandback} once more. 

\begin{proof} The proof is an application of the results in 
Theorem \ref{thm:weightsfrompartialisometryforJandJast}. Observe 
that by taking $f(\la) = \la^k v_j$, $g(\la) = v_i$, $0\leq i,j\leq L$, $k\in \N$, we get the $(i,j)$-th matrix entry of the integral on the left hand side  
is the $(i,j)$-th matrix entry of the integral on the right hand side. So we get, as an identity for matrix-valued integrals 
\begin{align*}
\int_0^\infty  W^{\nu}(\la) \, \la^{k+1}\, 
d\nu_\Tr(\la) =
\int_0^\infty  
\bigl( Y_0^\ast + Z_1^\ast \Phi_1(\la)\bigr) W^{\mu}(\la) 
\bigl( Y_0 +  \Phi_1(\la)^\ast Z_1\bigr) \, \la^k\, 
d\mu_\Tr(\la) 
\end{align*}
as an expression for the $k+1$-th matrix moment of the matrix measure $\mu$ in terms of the matrix moments for the matrix measure $\nu$.
Writing $ Y_0^\ast + Z_1^\ast \Phi_1(\la) = \Xi_0 + \la \Xi_1$ 
for the moment and $S^\nu_k = \int_0^\infty  W^{\nu}(\la) \, \la^{k}\, 
d\nu_\Tr(\la)$ for the moments, and similarly for $S^\mu_k$, we get 
\begin{align*}
S^\nu_{k+1} = \Xi_0 S^\mu_k \Xi_0^\ast + 
\Xi_0 S^\mu_{k+1} \Xi_1^\ast + \Xi_1 S^\mu_{k+1} \Xi_0^\ast 
+ \Xi_1 S^\mu_{k+2} \Xi_1^\ast.
\end{align*}
Dividing by $z^{k+1}$, summing over $k\in \N$ and interchanging 
summation and integration, which is valid for $|z|\gg 0$, and relabelling, we get 
\begin{align*}
\sum_{k=1}^\infty \frac{S^\nu_k}{z^k} = 
\Xi_0 \Bigl( \frac{1}{z} \sum_{k=0}^\infty \frac{S^\mu_k}{z^k} 
\Bigr) \Xi_0^\ast  
+ \Xi_0 \Bigl( \sum_{k=1}^\infty \frac{S^\mu_k}{z^k}  \Bigr) \Xi_1^\ast
+ \Xi_1 \Bigl( \sum_{k=1}^\infty \frac{S^\mu_k}{z^k}  \Bigr) \Xi_0^\ast
+ \Xi_1 \Bigl( z \sum_{k=2}^\infty \frac{S^\mu_k}{z^k}  \Bigr) \Xi_1^\ast. 
\end{align*}
Recall $M^\nu(z) = \frac{-1}{z} \sum_{k=0}^\infty \frac{S^\nu_k}{z^k}$ for $|z|\gg 0$, we can rewrite this in terms of the Stieltjes transforms taking into account $S_0^\mu=\Id= S^\nu_0$. 
This gives $zM^\nu(z)+\Id$
\begin{align*}
 &\, =
\Xi_0 M^\mu(z) \Xi_0^\ast  
+ \Xi_0 \Bigl( zM^\mu(z) + \Id  \Bigr) \Xi_1^\ast
+ \Xi_1 \Bigl( zM^\mu(z) + \Id \Bigr) \Xi_0^\ast
+ \Xi_1 \Bigl( z^2 M^\mu(z) + z\Id + S^\mu_1  \Bigr) \Xi_1^\ast\\ 
&\, =
\bigl( \Xi_0 +\Xi_1 z\bigr)  M^\mu(z) \bigl( \Xi_0^\ast + \Xi_1^\ast z\bigr)   
+ \Xi_0 \Xi_1^\ast
+ \Xi_1 \Xi_0^\ast
+ \Xi_1 \bigl( z\Id + S^\mu_1  \bigr) \Xi_1^\ast.
\end{align*}
Note that 
$\Phi_1(\la) = (A_0')^{-1}(\la-B_0')$ so that 
$\Xi_0= Y_0^\ast - Z_1^\ast (A_0')^{-1} B_0'$ and 
$\Xi_1= Z_1^\ast (A_0')^{-1}$. Moreover, $S_1^\mu= B_0'$ 
since $\int_0^\infty \Phi_1(\la) W^\mu(\la) \, d\mu_\Tr(\la) =0$. 
Plugging this in, and simplifying gives the first result. 

The second result is proved analogously. 
\end{proof}

The Stieltjes-Perron inversion formula then shows that the relation of the Stieltjes transforms in 
Corollary \ref{cor:thm:weightsfrompartialisometryforJandJast}
gives the following result. 

\begin{cor}\label{cor:thm:COR2weightsfrompartialisometryforJandJast}
For $a<b$ we have 
\begin{multline*}
 \int_a^b  W^{\nu}(\la) \, \la\, 
d\nu_\Tr(\la)  + \frac12\bigl( a\, \nu(\{a\}) + b\, \nu(\{b\})\bigr)  
 \\
= \int_a^b   
\bigl( Y_0^\ast + Z_1^\ast \Phi_1(\la)\bigr) W^{\mu}(\la) 
\bigl( Y_0 +  \Phi_1(\la)^\ast Z_1\bigr) \, 
d\mu_\Tr(\la) 
 \\ 
+ \frac12\Bigl( 
\bigl( Y_0^\ast + Z_1^\ast \Phi_1(a)\bigr) \mu(\{a\}) 
\bigl( Y_0 +  \Phi_1(a)^\ast Z_1\bigr) 
+
\bigl( Y_0^\ast + Z_1^\ast \Phi_1(b)\bigr) \mu(\{b\}) 
\bigl( Y_0 +  \Phi_1(b)^\ast Z_1\bigr) \Bigr)
\end{multline*}
and 
\begin{multline*}
\int_a^b W^{\mu}(\la) \, \la\, 
d\mu_\Tr(\la) 
+ \frac12\bigl( a\, \mu(\{a\}) + b\, \mu(\{b\})\bigr)  
\\ =
\int_a^b  
\bigl( Y_0 + X_0 \Psi_1(\la)\bigr) W^{\nu}(\la) 
\bigl( Y_0^\ast +  \Psi_1(\la)^\ast X_0^\ast \bigr)  
d\nu_\Tr(\la)
 \\ 
+\frac12\Bigl( 
\bigl( Y_0 + X_0 \Psi_1(a)\bigr) \nu(\{a\}) 
\bigl( Y_0^\ast +  \Psi_1(a)^\ast X_0^\ast \bigr) 
+ 
\bigl( Y_0 + X_0 \Psi_1(b)\bigr) \nu(\{b\}) 
\bigl( Y_0^\ast +  \Psi_1(b)^\ast X_0^\ast \bigr) 
\Bigr).
\end{multline*}
\end{cor}

Note that the result of Corollary \ref{cor:thm:COR2weightsfrompartialisometryforJandJast} is a result for matrix valued integrals. The result simplifies 
considerably if we assume that $a$ and $b$ do not correspond to 
discrete mass points of $\mu$ and $\nu$. 

\begin{proof}
We use the Stieltjes-Perron inversion formula for more general  (finite) complex Radon measures, see 
\cite[Thm.~A.13]{Schm}, and that the spectral measures for bounded operators give 
regular measures, see \cite[Def.~12.17]{Rudi} and the remarks made there. Note that the matrix entries then correspond to 
Stieltjes transforms for such a scalar complex Radon measure. For such a measure $\theta$ we have for a polynomial 
$f(x) = \sum_{p=0}^N a_px^p$ that 
\begin{equation}\label{eq:connectionStieltjestransform}
\int_\R \frac{f(x)}{x-z}\, d\theta(x) = 
f(z) \int_\R \frac{1}{x-z}\, d\theta(x) +
\sum_{r=0}^{N-1} \Bigl( \sum_{p=r+1}^N a_p m_{p-r-1}\Bigr)z^r 
\end{equation}
where $m_k = \int_\R x^k\, d\theta(x)$ are the moments. 
Using this in Corollary 
\ref{cor:thm:weightsfrompartialisometryforJandJast} for each of the matrix entries and the Stieltjes-Perron inversion formula, we obtain the result.
\end{proof}

%%%%%%%%%%%%%%%%%%%%%%%%%%%%%%%%%
%%%%%%%%%%%%%%%%%%%%%%%%%%%%%%%%%
%%%%%%%%%%%%%%%%%%%%%%%%%%%%%%%%%
%%%%%%%%%%%%%%%%%%%%%%%%%%%%%%%%%

\section{Symmetric case}\label{sec:symmetriccase}

In general, the three-term recurrence relation \eqref{eq:MOPPn} 
for the matrix orthogonal polynomials $P_n$ is not symmetric because of the occurrence of $S^{11}_n$ and $S^{22}_n$. 
In this section we assume additionally that $S^{11}_n = 0= S^{22}_n$ for all $n$. 
From the expression \eqref{eq:defS11andS22} we see that this is precisely the case if 
$b_{n,k}=0$ for $0\leq k<L$, and we call this the symmetric case.
This is in accordance with e.g. \cite{LourVA}. 
So the recurrence 
\eqref{eq:recursionscalarLorthopols} reduces to 
$zp_n(z) = a_n p_{n+1}(z) + b_{n,L} p_{n-L}(z)$ with $b_{n,L}=0$ for $L>n$ and 
$b_{n,L}\not=0$ for $n\geq L$. It follows that the matrices $X_p$, $Y_p$, $Z_p$ 
given in \eqref{eq:defXpYpZp} simplify considerably, as do the matrices $T^{21}_p$, 
$S^{12}_p$ in Lemma \ref{lem:rangekernelT}. 
The symmetric case is exactly the case where the matrix polynomials are (or are equivalent to) diagonal matrix polynomials. 

As observed, in the symmetric case the recurrence \eqref{eq:MOPPn} is symmetric, so that iterating gives 
\begin{equation*}
\begin{split}
z^2 P_{2n}(z) &\, =  S_n^{12}T^{21}_n P_{2(n+1)}(z) + 
\bigl( S_n^{12} (S_n^{12})^\ast + (T^{21}_{n-1})^\ast T^{21}_{n-1}\bigr) P_{2n}(z) 
+ (S_{n-1}^{12}T^{21}_{n-1})^\ast P_{2(n-1)}(z),\\
z^2 P_{2n+1}(z) &\, =  T^{21}_nS_{n+1}^{12} P_{2(n+1)+1}(z)
+ \bigl( T^{21}_n(T^{21}_n)^\ast +(S_n^{12})^\ast S_n^{12} \bigr) P_{2n+1}(z) 
\\ &\, \qquad \qquad\qquad \qquad
\qquad \qquad\qquad \qquad
\qquad\qquad \qquad
+ (T^{21}_{n-1}S_{n}^{12})^\ast P_{2(n-1)+1}(z).
\end{split}
\end{equation*}
Note that all the matrices involved in the three-term recurrence relation 
are diagonal. We generate diagonal matrix orthogonal 
polynomials $Q^e_n$ and $Q^o_n$, corresponding to the even and odd case respectively, satisfying
\begin{equation}\label{eq:defMOPQn}
\begin{split}
z Q^e_{n}(z) &\, =  S_n^{12}T^{21}_n Q^e_{n+1}(z) + 
\bigl( S_n^{12} (S_n^{12})^\ast + (T^{21}_{n-1})^\ast T^{21}_{n-1}\bigr) Q^e_n(z) 
+ (S_{n-1}^{12}T^{21}_{n-1})^\ast Q^e_{n-1}(z),\\
z Q^o_{n}(z) &\, =  T^{21}_nS_{n+1}^{12} Q^o_{n+1}(z)
+ \bigl( T^{21}_n(T^{21}_n)^\ast +(S_n^{12})^\ast S_n^{12} \bigr) Q^o_{n}(z) 
+ (T^{21}_{n-1}S_{n}^{12})^\ast Q^o_{n-1}(z)
\end{split}
\end{equation}
with initial conditions $Q^o_0(z)=\Id = Q^e_0(z)$, since 
$(S^{12}_{-1}T^{21}_{-1})^\ast = 0 = (T^{21}_{-1}S^{12}_0)^\ast$. These matrix polynomials are diagonal with scalar orthogonal polynomials on the diagonal. The polynomials 
are related to the even and odd degree polynomials, since  
$P_{2n}(z) = Q^e_{n}(z^2)$, which is diagonal, and 
$P_{2n+1}(z) = Q^o_{n}(z^2) P_1(z)= Q^o_{n}(z^2)(S_0^{12})^{-1} z$, which follows 
by comparing three-term recurrence relations and the initial values. Note that 
$P_{2n+1}(z)$ is not diagonal, since $S_0^{12}$ is not diagonal. 
However, we can transform to a diagonal situation by multiplying by a unitary matrix, cf. Section \ref{ssec:matrixOP}. 
Fix the permutation $\si= \begin{pmatrix}
0 & 1 & \cdots & L \\ L & 0 & \cdots & L-1\end{pmatrix}$, and let 
$U$ be the corresponding permutation $(L+1)\times (L+1)$-matrix
defined by $Uv_i = v_{\si(i)}$ or $U_{i,j}=1$ if $i=\si(j)$ and 
$U_{i,j}=0$ otherwise. Then $UU^\ast =\Id = U^\ast U$, so $U$ is unitary. 

\begin{prop}
Define the matrix polynomial family $(\widetilde{P}_n)_{n\in \N}$ by $\widetilde{P}_{2n} = P_{2n}$, 
$\widetilde{P}_{2n+1} = U P_{2n+1}$, then 
$(\widetilde{P}_{n})_{n\in \N}$ is an orthonormal family of diagonal matrix polynomials with respect to the diagonal matrix measure $\rho$.  
\end{prop}

\begin{proof} Since $U$ is unitary, the family 
$(\widetilde{P_n})_{n\in \N}$ is orthonormal, see Section \ref{ssec:matrixOP}, with respect to the same measure. Now 
$UP_{2n+1}(z) = UQ^o_n(z^2)U^{-1} U (S_0^{12})^{-1} z$
is diagonal, since $UQ^o_n(z^2)U^{-1}$ is diagonal and so 
is $U (S_0^{12})^{-1}$. Since $(\widetilde{P}_n)_{n\in \N}$ 
is diagonal, so is the weight. 
% ??do we need a better argument??
\end{proof}

Note that $(UQ^o_n(z^2)U^{-1})_{\si(i),\si(i)} 
= Q^o_n(z^2)_{i,i}$ 
and $U (S_0^{12})^{-1} =\textrm{diag} (a_0^{-1}, \cdots, a_{L-1}^{-1}, b_{L,L}^{-1})$.

We make the diagonal matrices in the recurrence \eqref{eq:defMOPQn} explicit. 
For the even polynomials we have 
\begin{equation}\label{eq:matricesevenMOPQn}
\begin{split}
(S^{12}_n T^{21}_n)_{i,j} 
&\, =\de_{i,j} 
\begin{cases}
a_{n(L+1)+i} \overline{ b_{(n+1)(L+1)+i, L}}, & 0\leq i <L, \\
a_{n(L+1)+L} \overline{ b_{n(L+1)+L, L}}, & i=L,
\end{cases}
\\
\bigl( S_n^{12} (S_n^{12})^\ast + (T^{21}_{n-1})^\ast T^{21}_{n-1}\bigr)_{i,j} &\, = 
\de_{i,j} 
\begin{cases}
a_{n(L+1)+i}^2 + |b_{(n+1)(L+1)+i,L}|^2, & 0\leq i <L, \\
a_{n(L+1)+L}^2 +|b_{n(L+1)+L,L}|^2, & i=L.
\end{cases}
\end{split}
\end{equation}
For the odd polynomials we have 
\begin{equation}\label{eq:matricesoddMOPQn}
\begin{split}
(T^{21}_n S^{12}_{n+1} )_{i,j} 
&\, =\de_{i,j} 
\begin{cases}
a_{n(L+1)+L}\overline{b_{(n+1)(L+1)+L,L}}, & i=0, \\
a_{(n+1)(L+1)+i-1}\overline{b_{(n+1)(L+1)+i-1,L}}, & 0<i\leq L,
\end{cases}
\\
\bigl( T^{21}_n(T^{21}_n)^\ast  + (S_n^{12})^\ast S_n^{12}  \bigr)_{i,j} &\, = 
\de_{i,j} 
\begin{cases}
a_{n(L+1)+L}^2 + |b_{n(L+1)+L,L}|^2, & i=0, \\
a_{n(L+1)+i-1}^2 +|b_{(n+1)(L+1)+i-1,L}|^2, & 0<i \leq L.
\end{cases}
\end{split}
\end{equation}

In the symmetric case, the matrices in the recurrence  
\eqref{eq:defPhiPsiMOP} are also diagonal matrices as can be seen by a straightforward calculation. In particular, $A'_p$ and $A''_p$ are given completely by their diagonals as in Lemma \ref{lem:ApaccentApdoubleaccentinvertible} and 
\begin{equation}\label{eq:BaccentpandBdoubleaccentpinsymmetriccase}
\begin{split}
(B'_p)_{i,j} &\, = \de_{i,j} \bigl( a^2_{p(L+1)+i} + |b_{p(L+1)+i,L}|^2\bigr) , \\
(B''_p)_{i,j} &\, = \de_{i,j} 
\begin{cases}
a^2_{(p-1)(L+1)+L} +|b_{p(L+1)+L,L}|^2, & i=0,  \\
a^2_{p(L+1)+i-1} +|b_{(p+1)(L+1)+i-1,L}|^2, & 0 <i \leq L.
\end{cases}
\end{split}
\end{equation}
This implies that the matrix polynomials $\Phi_n(z)$ and $\Psi_n(z)$ are 
completely reducible with only non-zero diagonal entries, which are 
orthogonal scalar polynomials. 

\begin{prop}\label{prop:relationQnPhinPsininsymmetriccase} 
In the symmetric case, the polynomials 
defined by \eqref{eq:defPhiPsiMOP} and \eqref{eq:defMOPQn} are related by  
\begin{align*} 
& Q^e_n(z)_{i,i}= \Phi_n(z)_{i,i}\quad \text{for}\quad 0\leq i <L,\\
& Q^o_n(z)_{i,i}= \Psi_n(z)_{i,i} \quad \text{for}\quad 0<  i \leq L,\\  
& \Psi_n(z)_{0,0} = Q_n^e(z)_{L,L}\quad \text{and}\quad \Phi_n(z)_{L,L} = Q_n^o(z)_{0,0}.
\end{align*}
\end{prop}

\begin{proof} This can be proved by comparing the three-term recurrence relations 
using \eqref{eq:matricesevenMOPQn}, \eqref{eq:matricesoddMOPQn} and 
\eqref{eq:BaccentpandBdoubleaccentpinsymmetriccase} and Lemma \ref{lem:ApaccentApdoubleaccentinvertible}. 

Alternatively, Proposition \ref{prop:relationPnPhinPsin} and 
$P_{2n}(z) = Q^e_{n}(z^2)$, 
$P_{2n+1}(z) = Q^o_{n}(z^2)(S_0^{12})^{-1} z$
relate $Q^e_n$ and $Q_n^o$ to $\Phi_n$ and $\Psi_n$. Since the matrix polynomials
$Q^e_n$, $Q_n^o$, $\Phi_n$ and $\Psi_n$ are diagonal, the result follows. 
\end{proof}

The results of Corollary \ref{cor:lem:JonFeigenfunctionJastJandback} lead to results for matrix polynomials with diagonal entries, and hence these reduce to much simpler relations. We find connection formulas for the diagonal matrix polynomials in the symmetric case.
As an example, the first result of Corollary \ref{cor:lem:JonFeigenfunctionJastJandback} only leads to non-zero identities for 
the $(i+1,i)$-th entry for $0\leq i <L$ and the $(0,L)$-entry. This gives
\begin{equation}
\begin{split}
\overline{b_{(n+1)(L+1)+i}} \Phi_{n+1}(z)_{i,i} + a_{n(L+1)+i}\Phi_{n}(z)_{i,i} 
&\,= \Psi_n(z)_{i+1,i+1} \bigl( \overline{b_{L+1+i,L}} \Phi_1(z)_{i,i}+ a_i \bigr), \\  
\overline{b_{n(L+1)+L,L}} \Phi_{n}(z)_{L,L} + a_{(n-1)(L+1)+L}\Phi_{n-1}(z)_{L,L} 
&\,= \Psi_n(z)_{0,0}\, \overline{b_{L,L}}, 
\end{split}
\end{equation}
where the first equation holds for $0\leq i <L$. The second result of Corollary \ref{cor:lem:JonFeigenfunctionJastJandback}   to non-zero entries only for the $(i,i+1)$-the entry for $0\leq i <L$ and for the  
$(L,0)$-th entry. This gives 
\begin{equation}
\begin{split}
a_{n(L+1)+i} \Psi_{n+1}(z)_{i+1,i+1} + b_{n(L+1)+i,L}\Psi_{n}(z)_{i+1,i+1} 
&\,= \Phi_n(z)_{i,i} \, a_i, \\  
a_{n(L+1)+L} \Psi_{n+1}(z)_{0,0} + b_{n(L+1)+L,L}\Psi_{n}(z)_{0,0} 
&\,= \Phi_n(z)_{L,L} \bigl( a_L \Psi_1(z)_{0,0} + b_{L,L} \bigr),  
\end{split}
\end{equation}
where the first equation holds for $0\leq i <L$. 

In the symmetric case, the recursion \eqref{eq:MOPPn} specialises to the recursion
\begin{equation}\label{eq:symmcaseassociatedPn}
\de_{n,0}\Id + zP^{(1)}_n(z) = A_n P^{(1)}_{n+1}(z) + A_{n-1}^\ast P^{(1)}_{n-1}(z),
\qquad n\in \N,
\end{equation}
for the associated polynomials $P^{(1)}_n$. Iterating 
\eqref{eq:symmcaseassociatedPn} gives 
\begin{equation}\label{eq:symmcaseassociatedPn2}
\de_{n,0} z \Id + \de_{n,1}A_0^\ast + z^2 P^{(1)}_n(z) = 
A_n A_{n+1}P^{(1)}_{n+2}(z) + \bigl( A_nA_n^\ast + A_{n-1}^\ast A_{n-1}\bigr) 
P^{(1)}_{n}(z) + A_{n-1}^\ast A_{n-2}^\ast P^{(1)}_{n-2}(z)
\end{equation}
and restricting to the case $n=2p$ and $n=2p+1$ then links to the 
associated polynomials $Q^{e,(1)}_p$ and $Q^{o,(1)}_p$ which 
solve \eqref{eq:defMOPQn} with initial conditions 
$Q^{e,(1)}_0(z)=0=Q^{o,(1)}_0(z)$ and 
$Q^{e,(1)}_1(z)=(A_0A_1)^{-1}=(S^{12}_0T^{21}_0)^{-1}$ and $Q^{o,(1)}_1(z)=(A_1A_2)^{-1} = (T^{21}_0 S^{12}_0)^{-1}$.
Note that the associated polynomials 
$Q^{e,(1)}_p$ and $Q^{o,(1)}_p$ are diagonal. 

\begin{prop}\label{prop:relationassociatedQnPhinPsininsymmetriccase} 
In the symmetric case, the associated polynomials  are related by  
\begin{equation*}
\widetilde{P}^{(1)}_{2p}(z) = Q^{e,(1)}_p(z^2)z, 
\qquad 
\widetilde{P}^{(1)}_{2p+1}(z) = UQ^o_p(z^2)U^{-1} U (S^{12}_0)^{-1} + UQ^{o,(1)}_p(z^2)U^{-1} U (S^{12}_0)^\ast  
\end{equation*}
as diagonal matrix polynomials. 
\end{prop}

Note that Proposition \ref{prop:relationassociatedQnPhinPsininsymmetriccase} 
shows that the associated polynomials $P^{(1)}_{2p}(z)$ of even degree are diagonal, but the 
 associated polynomials $P^{(1)}_{2p+1}(z)$ of odd degree are not diagonal. 

\begin{proof} Taking $n=2p$ in \eqref{eq:symmcaseassociatedPn2} gives 
\begin{equation*}
\de_{p,0}z\Id + z^2P^{(1)}_{2p}(z) 
= A_{2p}A_{2p+1} P^{(1)}_{2(p+1)}(z) 
+ \bigl( A_{2p}A^\ast_{2p}+ A^\ast_{2p-1}A_{2p-1}\bigr) P^{(1)}_{2p}(z)
+ A_{2p-1}^\ast A_{2p-2}^\ast P^{(1)}_{2(p-1)}(z)
\end{equation*}
and comparing with the recurrence for the associated polynomials 
$Q^{e,(1)}_p(w)$ with $w=z^2$ gives the first result. 

Taking $n=2p+1$ in \eqref{eq:symmcaseassociatedPn2} gives 
\begin{multline*}
\de_{p,0}A_0^\ast + z^2P^{(1)}_{2p+1}(z) 
= \\ A_{2p+1}A_{2p+2} P^{(1)}_{2(p+1)+1}(z) 
+ \bigl( A_{2p+1}A^\ast_{2p+1}+ A^\ast_{2p}A_{2p}\bigr) P^{(1)}_{2p+1}(z)
+ A_{2p}^\ast A_{2p-1}^\ast P^{(1)}_{2(p-1)+1}(z)
\end{multline*}
and comparing with the recurrence for the polynamials $Q^o_p(w)$ and 
its associated $Q^{o,(1)}_p(w)$ with $w=z^2$ gives 
$P^{(1)}_{2p+1}(z) = Q^o_p(z^2)C + Q^{o,(1)}_p(z^2)D$ for suitable 
matrices $C$ and $D$. Taking $p=0$ gives 
$A_0^{-1} = P^{(1)}_{1}(z) = C$, and next taking $p=1$ 
gives 
$P^{(1)}_{3}(z) = Q^o_1(z^2)A_0^{-1} + Q^{o,(1)}_1(z^2)D$
and plugging in $P^{(1)}_{3}(z) = A_2^{-1}(z^2A_1^{-1}-A_1^\ast)A_0^{-1}$ 
and $Q^o_1(z^2)= A_2^{-1}A_1^{-1}(z^2\Id -(A_1A_1^\ast + A_0^\ast A_0))$ 
gives $D=A_0^\ast$. 
Now $A_0=S^{12}_0$ gives the result for the polynomials 
$P_n$. Now multiply with either the identity or the unitary $U$.
\end{proof}

In the symmetric case the Stieltjes transforms of the scalar orthogonal polynomials on the diagonal can be connected by specialising the results in Section \ref{sec:relationbetweenMOP} and Section \ref{sec:polardecompositions}, see also 
Section \ref{sec:klr-constant-symmetric} for an explicit case.

%%%%%%%%%%%%%%%%%%%%%%%%%%%%%%%%%
%%%%%%%%%%%%%%%%%%%%%%%%%%%%%%%%%
%%%%%%%%%%%%%%%%%%%%%%%%%%%%%%%%%
%%%%%%%%%%%%%%%%%%%%%%%%%%%%%%%%%

\section{The symmetric constant case for \texorpdfstring{$L=2$}{L=2}}
\label{sec:klr-constant-symmetric}

In this section we consider the symmetric case of Section \ref{sec:symmetriccase} with
constant coefficients and $L=2$. It should be clear from this section that the approach can be extended to more general $L$. 
We determine the polynomials $\Phi_n$
and $\Psi_n$, their orthogonality measures and their Stieltjes transforms,
and verify the relations obtained in Sections \ref{sec:relationbetweenMOP}
and \ref{sec:polardecompositions}.

Let $a>0$ and $b\in\R\setminus\{0\}$. We set
\[
 a_n=a,\qquad b_{n,0}=b_{n,1}=0\quad(n\in\N),
 \qquad b_{n,2}=b\quad(n\geq2).
\]
The recurrence
\eqref{eq:recursionscalarLorthopols} becomes
$zp_n(z) = ap_{n+1}(z)$ for $n\in \{0,1\}$ and 
\[
zp_n(z)=ap_{n+1}(z)+bp_{n-2}(z),\qquad n\in\N_{\geq 2},
\]
with initial value $p_0(z)=1$.
Substituting these coefficients into \eqref{eq:JJastonF} and
\eqref{eq:JastJonF}, we obtain
\begin{equation}\label{eq:klr-constant-coefficients}
 \begin{aligned}
 A'_n=A''_n&=ab\Id, &&n\geq0,\\
 B'_n=B''_n&=(a^2+b^2)\Id, &&n\geq1,
 \end{aligned}
\end{equation}
with initial diagonal blocks
\[
 B'_0 =\operatorname{diag}(a^2,a^2,a^2+b^2),\qquad 
 B''_0 =\operatorname{diag}(b^2,a^2+b^2,a^2+b^2).
\]
Note that this shows that $JJ^\ast-J^\ast J$ is a finite rank operator (of rank $2$), and so $J$ is essentially normal, see \cite{Davi}. 

In order to study the spectral decomposition of $JJ^\ast$ and 
$J^\ast J$ we need a perturbation of the Chebyshev polynomials. 
Let $U_n$ denote the Chebyshev polynomials of the second kind. Then
\[
 U_{-1}(x)=0,\qquad U_0(x)=1,\qquad
 2xU_n(x)=U_{n+1}(x)+U_{n-1}(x),\quad n\geq0.
\]
In order to write to relate this recurrence relation to \eqref{eq:defPhiPsiMOP}, we set
\begin{equation}\label{eq:klr-constant-variable}
 x(\la)=\frac{\la-a^2-b^2}{2ab}.
\end{equation}
Substituting \eqref{eq:klr-constant-coefficients} into
\eqref{eq:defPhiPsiMOP}, dividing by $ab$, and combining the terms of degree $n$, we obtain
\begin{equation}\label{eq:klr-constant-recurrences}
 \begin{aligned}
 2x(\la)\Phi_n(\la)&=\Phi_{n+1}(\la)+\Phi_{n-1}(\la),\\
 2x(\la)\Psi_n(\la)&=\Psi_{n+1}(\la)+\Psi_{n-1}(\la),
 \end{aligned}
\end{equation}
for  $n\geq1$. Since $\Phi_0=\Psi_0=\Id$, the recurrence at $n=0$ and the initial
diagonal blocks give
\[
 \begin{aligned}
 \Phi_1(\la)
 &=\frac{\la\Id-B'_0}{ab}
 =2x(\la)\Id+\frac ba\operatorname{diag}(1,1,0),\\
 \Psi_1(\la)
 &=\frac{\la\Id-B''_0}{ab}
 =2x(\la)\Id+\frac ab\operatorname{diag}(1,0,0).
 \end{aligned}
\]

Both $U_n(x)$ and $U_{n-1}(x)$ satisfy the recurrence
\eqref{eq:klr-constant-recurrences}. Comparing the initial values at
$n=0,1$, we obtain
\begin{equation}\label{eq:klr-constant-polynomials}
 \begin{aligned}
 \Phi_n(\la)
 &=U_n(x(\la))\Id
   +\frac ba\,U_{n-1}(x(\la))\operatorname{diag}(1,1,0),\\
 \Psi_n(\la)
 &=U_n(x(\la))\Id
   +\frac ab\,U_{n-1}(x(\la))\operatorname{diag}(1,0,0),
 \end{aligned}
\end{equation}
for $n\geq0$.  

Note that if we replace $b$ by $-b$, then $x(\la)$ changes to $-x(\la)$.
Since $U_n(-x)=(-1)^nU_n(x)$, the formulas
\eqref{eq:klr-constant-polynomials} show that both $\Phi_n$ and $\Psi_n$
are multiplied by $(-1)^n$.

The matrix measures follow from the orthogonality
relations for the scalar orthogonal polynomials that are combinations of the Chebyshev polynomials of the form 
$q_n(x)=U_n(x)-2rU_{n-1}(x)$ for several values of $r$ using 
\eqref{eq:klr-constant-polynomials}.
We write $\delta_y$ for the unit point mass at $y$.

\begin{lemma}\label{lem:klr-constant-chebyshev}
Let $r\in\R$. The polynomials
\[
 q_n(x)=U_n(x)-2rU_{n-1}(x),\qquad n\geq0,
\]
are orthonormal with respect to the probability measure
\begin{equation}\label{eq:klr-constant-scalar-measure}
 \frac2\pi\frac{\sqrt{1-x^2}}{(1-4rx+4r^2)}\mathbf1_{(-1,1)}(x)\,dx
 +\begin{cases}
 \displaystyle\left(1-\frac1{4r^2}\right)\delta_{r+1/(4r)},
                       &|r|>\frac12,\\[2pt]
 0,                    &|r|\leq\frac12.
 \end{cases}
\end{equation}
\end{lemma}

Note that for $r=\tfrac12$ and $r=-\tfrac12$, the continuous part of
\eqref{eq:klr-constant-scalar-measure} is a Jacobi weight with
parameters $(\alpha,\beta)=(-\tfrac12,\tfrac12)$ and
$(\tfrac12,-\tfrac12)$, respectively. The corresponding $q_n$ are
known as the Chebyshev polynomials of the third and fourth kind. 

\begin{proof}
Upon comparing  with the Askey--Wilson polynomials
in \cite[(4.28)-(4.32)]{AskeW1985} for $q=0$, we see that $(q_n)_{n\geq0}$ is their
specialization with $q=0$ and parameters $(2r,0,0,0)$.
The orthogonality measure \eqref{eq:klr-constant-scalar-measure} for
$|r|\ne\tfrac12$ follows from \cite[(4.32)]{AskeW1985}.
The cases $r=\pm\tfrac12$ follow by dominated convergence from
$|r|<\tfrac12$; the point mass is absent at the limit.
This completes the proof.
\end{proof}

To state the matrix measures, we put
\[
 \la_-=(a-|b|)^2,\qquad \la_+=(a+|b|)^2,
\]
\[
 h(\la)=\sqrt{(\la_+-\la)(\la-\la_-)},
 \qquad \la_-<\la<\la_+,
\]
and write $u_+=\max\{u,0\}$ for $u\in\R$.

\begin{prop}\label{prop:klr-constant-measures}
The orthogonality measures for the matrix polynomials in  \eqref{eq:klr-constant-polynomials},
normalised by $\mu(\R)=\nu(\R)=\Id$, are
\begin{align*}
 d\mu(\la)&=\frac{h(\la)}{2\pi a^2b^2}
 \operatorname{diag}\left(\frac{a^2}{\la},\frac{a^2}{\la},1\right)
 \mathbf1_{(\la_-,\la_+)}(\la)\,d\la + \left(1-\frac{a^2}{b^2}\right)_+
      \operatorname{diag}(1,1,0)\,\delta_0,
\\[3pt]
 d\nu(\la)&=\frac{h(\la)}{2\pi a^2b^2}
 \operatorname{diag}\left(\frac{b^2}{\la},1,1\right)
 \mathbf1_{(\la_-,\la_+)}(\la)\,d\la +\left(1-\frac{b^2}{a^2}\right)_+
      \operatorname{diag}(1,0,0)\,\delta_0.
\end{align*}
\end{prop}

Note that the measures $\mu$ and $\nu$ cannot both have a discrete mass point at $0$. The matrix measure $\mu$ has a discrete mass point if $a^2<b^2$ and $\nu$ has a discrete mass point if $a^2>b^2$. 

\begin{proof}
From \eqref{eq:klr-constant-variable}, we have
\[
 \sqrt{1-x(\la)^2}=\frac{h(\la)}{2a|b|},\qquad
 \left|\frac{dx}{d\la}\right|=\frac1{2a|b|}.
\]
Thus the Chebyshev measure becomes
$h(\la)\,d\la/(2\pi a^2b^2)$ for either sign of $b$.
The first two entries of $\Phi_n$ correspond to $r=-b/(2a)$ in
Lemma~\ref{lem:klr-constant-chebyshev}, whereas the first entry of
$\Psi_n$ corresponds to $r=-a/(2b)$. The denominators of \eqref{eq:klr-constant-scalar-measure} simplify to
\[
 1+2\frac ba x(\la)+\frac{b^2}{a^2}=\frac{\la}{a^2},\qquad
 1+2\frac ab x(\la)+\frac{a^2}{b^2}=\frac{\la}{b^2}.
\]
This gives the continuous parts of the measures. The possible point
mass in either case is at
\[
 x=-\frac12\left(\frac ab+\frac ba\right),\qquad \la=0.
\]
The condition $|r|>1/2$ and the weight $1-1/(4r^2)$ give the
mass matrices. In particular, $\mu$ has a nonzero mass at zero exactly
when $a<|b|$, and $\nu$ has one exactly when $a>|b|$.
Scalar orthogonality in each diagonal entry now gives
\[
 \int\Phi_n(\la)\,d\mu(\la)\,\Phi_m(\la)^\ast=\de_{n,m}\Id,
 \qquad
 \int\Psi_n(\la)\,d\nu(\la)\,\Psi_m(\la)^\ast=\de_{n,m}\Id.
\]
This completes the proof.
\end{proof}

The measures $\mu$ and $\nu$ depend only on $a$ and $|b|$.
If $a=|b|$, both measures have support $[0,4a^2]$ and no point mass.
For \(0<\lambda<4a^2\), the first two diagonal entries of \(\mu\) and the first diagonal entry of \(\nu\) have density
\[
\frac{1}{2\pi a^2}\sqrt{\frac{4a^2-\lambda}{\lambda}},
\]
while the remaining diagonal entries have density
\[
\frac{\sqrt{\lambda(4a^2-\lambda)}}{2\pi a^4}.
\]
The singularity at zero in the first density is integrable.

We next compute the Stieltjes transforms, using the convention of
Section \ref{ssec:differenceeqtnmatrixOP}. For the unmodified Chebyshev
measure, we put
\[
 m_0(z)=\frac1{2\pi a^2b^2}
 \int_{\la_-}^{\la_+}\frac{h(\la)}{\la-z}\,d\la.
\]
The change of variable $\la=a^2+b^2+2a|b|x$ reduces this integral to
the Stieltjes transform of the normalised Chebyshev weight of the second
kind, see \cite[Table~18.3.1]{NIST}.
Integrating the generating function \cite[18.12.10]{NIST} against this
weight, we obtain
\begin{equation}\label{eq:klr-constant-m0}
 m_0(z)=\frac{a^2+b^2-z+\sqrt{(z-\la_-)(z-\la_+)}}{2a^2b^2}.
\end{equation}
Here the square root is analytic on \(\mathbb C\setminus[\lambda_-,\lambda_+]\), and its asymptotic behavior at infinity is
\[\sqrt{(z-\lambda_-)(z-\lambda_+)}
=z-a^2-b^2+O(z^{-1}),\qquad z\to\infty.
\]
Let
\[
 d\sigma(\lambda)=\frac{h(\lambda)}{2\pi a^2b^2}
 \mathbf{1}_{(\lambda_-,\lambda_+)}(\lambda)\,d\lambda,
\]
so that
$m_0(z)=\int_{\mathbb{R}}(\lambda-z)^{-1}\,d\sigma(\lambda)$.
The formulas for $\mu$ and $\nu$ give
\[
 \lambda\,d\mu_{11}(\lambda)=a^2\,d\sigma(\lambda),
 \qquad
 \lambda\,d\nu_{11}(\lambda)=b^2\,d\sigma(\lambda),
\]
since multiplication by $\lambda$ cancels the denominators in
their densities and removes the possible masses at zero.
Integrating against $(\lambda-z)^{-1}$, using
$\mu_{11}(\mathbb{R})=\nu_{11}(\mathbb{R})=1$ and
$\lambda/(\lambda-z)=1+z/(\lambda-z)$, we obtain
\[
 a^2m_0(z)=1+zM^\mu_{11}(z),
 \qquad
 b^2m_0(z)=1+zM^\nu_{11}(z).
\]
Hence, the matrix Stieltjes transforms are given by
\begin{equation}\label{eq:klr-constant-transforms}
 \begin{aligned}
 M^\mu(z)&=\operatorname{diag}\bigl(f(z),f(z),m_0(z)\bigr),\\
 M^\nu(z)&=\operatorname{diag}\bigl(g(z),m_0(z),m_0(z)\bigr),
 \end{aligned}
\end{equation}
where
\[
 f(z)=\frac{a^2m_0(z)-1}{z},\qquad
 g(z)=\frac{b^2m_0(z)-1}{z},\qquad
 z\in\C\setminus[0,\infty).
\]
These formulas include the possible point masses at zero.
Indeed, an atom of weight $w$ at zero contributes $-w/z$ to
the corresponding scalar Stieltjes transform, so its weight
is recovered by multiplying the transform by $-z$ and taking
the limit as $z\uparrow0$ along the negative real axis.
In the present case,
\[
 -zf(z)=1-a^2m_0(z),\qquad
 -zg(z)=1-b^2m_0(z).
\]
Taking this limit in \eqref{eq:klr-constant-m0}, we obtain
\[
 \lim_{z\uparrow0}m_0(z)=\frac1{\max\{a^2,b^2\}}.
\]
By dominated convergence, $-zM^\mu(z)$ and $-zM^\nu(z)$ tend
to the respective mass matrices at zero: the factor
$(-z)/(\la-z)$ is bounded by $1$ for $\la\geq0$ and $z<0$,
and converges to $\mathbf1_{\{0\}}(\la)$.
Thus \eqref{eq:klr-constant-transforms} gives
\[
 \begin{aligned}
 \mu(\{0\})
 &=\lim_{z\uparrow0}\bigl(-zM^\mu(z)\bigr)
 =\left(1-\frac{a^2}{b^2}\right)_+
   \operatorname{diag}(1,1,0),\\
 \nu(\{0\})
 &=\lim_{z\uparrow0}\bigl(-zM^\nu(z)\bigr)
 =\left(1-\frac{b^2}{a^2}\right)_+
   \operatorname{diag}(1,0,0).
 \end{aligned}
\]
These agree with Proposition \ref{prop:klr-constant-measures};
in particular, both masses vanish when $a=|b|$.

Substituting \eqref{eq:klr-constant-m0} into $f$ and $g$, and writing
\[
 R(z)=\sqrt{(z-\la_-)(z-\la_+)},
\]
with the branch specified above, we obtain the explicit formulas
\begin{align*}
 M^\mu(z)
 &=\operatorname{diag}\left(
 \frac{a^2-b^2-z+R(z)}{2b^2z},
 \frac{a^2-b^2-z+R(z)}{2b^2z},
 \frac{a^2+b^2-z+R(z)}{2a^2b^2}
 \right),\\
 M^\nu(z)
 &=\operatorname{diag}\left(
 \frac{b^2-a^2-z+R(z)}{2a^2z},
 \frac{a^2+b^2-z+R(z)}{2a^2b^2},
 \frac{a^2+b^2-z+R(z)}{2a^2b^2}
 \right),
\end{align*}
for  $z\in\C\setminus[0,\infty)$.

%%%%%%%%%%%%%%%%%%%%%%%%%%%%%%%%%
%%%%%%%%%%%%%%%%%%%%%%%%%%%%%%%%%
%%%%%%%%%%%%%%%%%%%%%%%%%%%%%%%%%
%%%%%%%%%%%%%%%%%%%%%%%%%%%%%%%%%

\end{document}